%% file: main.tex
\documentclass[11pt,letterpaper,reqno]{amsart}
\usepackage{graphicx} 
\usepackage{amsthm,amsfonts,amssymb,amsmath,mathrsfs}
\usepackage{xcolor}
\usepackage{cleveref}
\usepackage{comment}
\usepackage{enumitem}
\usepackage{url}

\newtheorem{theorem}{Theorem}
\newtheorem{corollary}{Corollary}
\newtheorem{proposition}{Proposition}
\newtheorem{lemma}{Lemma}

\newtheorem{thmx}{Theorem}

\theoremstyle{remark}
\newtheorem{remark}{Remark}

\theoremstyle{definition}

\renewcommand{\Re}{\mathrm{Re} \,}
\renewcommand{\Im}{\mathrm{Im} \,}


\title{Quadratic cones on which few harmonic functions vanish}
\author{Josef E. Greilhuber}
\date{\today}

\begin{document}

\begin{abstract}
    We show that, in dimension three and higher, the space of harmonic functions vanishing on the cone defined by a generically chosen harmonic quadratic polynomial is two-dimensional.
    This phenomenon is surprisingly robust, generalizing to arbitrary elliptic differential operators of second order, with the cone replaced by the level set of a solution at a nondegenerate critical value. As long as the tangent cone to the level set at the critical point satisfies a certain genericity condition, the space of solutions vanishing on the level set is at most two-dimensional.
\end{abstract}

\maketitle

\section{Introduction}

\input{Introduction}

\section{Preliminary observations}
\label{sec:preliminaries}

\input{Preliminary_Observations}

\section{Conical coordinates and Niven's equations}
\label{sec:niven_equations}

\input{Conical_Coordinates_and_Brelot_Choquet}

\input{Main_Section}

\section{WKB approximation of the associated Legendre polynomials}
\label{sec:WKB}

\input{WKB}

\section{Zero sets of solutions to elliptic PDEs: Proof of \Cref{thm:general}}
\label{sec:smooth}

\input{Smooth_Functions}

\bibliography{references}
\bibliographystyle{plain}

\end{document}

%% file: Introduction.tex
Given a subset $\mathcal C \subseteq \mathbb R^d$ and a point $p \in \mathcal C$, one may ask for the dimension of the space of (real valued) harmonic functions defined near $p$ which vanish on $\mathcal C$. Despite the apparent simplicity of this question, its answer seems not to be known in all but a few cases where the answer is either zero or infinite.

In two dimensions, it is well known and easy to see that this is always the case.\footnote{Any harmonic function $u$ on $\mathbb R^2$ is the imaginary part of a holomorphic function $h$. If $u$ vanishes on $\mathcal C$, $h$ is real valued there. Thus, $\{\Im h^n\}_{n=1}^\infty$ yields an infinite number of linearly independent harmonic functions which likewise vanish on $\mathcal C$, unless $u \equiv 0$.} 
For $d \geq 3$, the picture is less clear-cut. The dimension in question can still be infinite, e.g.\ if $\mathcal C$ is a real-analytic hypersurface \cite[Example 4.1]{LogunovMalinnikova2016}, a Coxeter system of hyperplanes through $p$ \cite[Lemma 4.4]{Agranovsky2000} or the cone $\{x^2+y^2+z^2-3w^2 = 0\}$ in $\mathbb R^4$ through $p=0$ \cite[Example 4.3]{LogunovMalinnikova2016}.

On the other hand, Agranovsky and Krasnov \cite[Section 6.1]{AgranovskyKrasnov2000} conjectured that there exists a quadratic cone on which a nonzero, but finite number of linearly independent harmonic functions vanish. In fact, Agranovsky \cite{Agranovsky2000} conjectured this is already the case for the cone $\{x^2+y^2-2z^2=0\}$ in $\mathbb R^3$, and that any harmonic function which vanishes on this cone is of the form $(x^2+y^2-2z^2)(A + Bxyz + Cz(x^2-y^2))$ for $A,B,C \in \mathbb R$. Mangoubi and Weller Weiser \cite{MangoubiWellerWeiser} provided strong evidence in support of this conjecture, but it remains open. The aim of this paper is to prove Agranovsky--Krasnov's conjecture is true for \emph{generic} cones.

Given $d$ nonzero real numbers $a_1,\ldots,a_d$, we consider the quadratic cone
\begin{align*}
    \mathcal C_a = \left\{ x \in \mathbb R^n: \frac{x_1^2}{a_1} + \ldots + \frac{x_d^2}{a_d} = 0 \right\}.
\end{align*}
Since we are interested in the set of harmonic functions vanishing on $\mathcal C_a$, it is natural to consider those $a$ with $\frac{1}{a_1} + \ldots + \frac{1}{a_d} = 0$. In this case, $\frac{x_1^2}{a_1} + \ldots + \frac{x_d^2}{a_d}$ is a nontrivial harmonic function vanishing on $\mathcal C_a$. It turns out that there is always another linearly independent harmonic function vanishing on $\mathcal C_a$, namely $x_1\ldots x_d(  \frac{x_1^2}{a_1} + \ldots + \frac{x_d^2}{a_d} )$. We will show that, in a certain sense, \emph{almost all} such cones admit no further harmonic functions vanishing on them. To state our result, let $\mathcal A^d = \{a \in (\mathbb R \setminus \{0\}) ^{d}: \frac{1}{a_1} + \ldots + \frac{1}{a_d} = 0\}$.

\begin{theorem}
\label{thm:main}
Let $d \geq 3$. There exists a countable intersection $\mathcal S$ of nonempty, Zariski-open subsets of $\mathcal A^d$ such that for all $a \in \mathcal S$, the following holds:

Any real valued harmonic function defined on a neighborhood of the origin and vanishing on the cone $\mathcal C_a$
must be of the form
\begin{align*}
    u(x) = \left( \frac{x_1^2}{a_1} + \ldots + \frac{x_d^2}{a_d} \right)\left(A + B\, x_1 \ldots x_d\right)
\end{align*}
for some constants $A,B \in \mathbb R$.
\end{theorem}

\begin{remark}
    In \Cref{thm:main}, \emph{Zariski-open} means open in the subset topology on $\mathcal A^d$ induced by the Zariski topology on $\mathbb R^d$. A countable intersection of non-empty, Zariski-open subsets is generic in a very strong sense. In particular, the set $\mathcal S$ is both of full Lebesgue measure and comeagre in $\mathcal A^d$.
\end{remark}

\begin{remark}
    Since Laplace eigenfunctions on the unit sphere $\mathbb S^d$ are in one-to-one correspondence with homogeneous harmonic polynomials, \Cref{thm:main} immediately implies the assertion that for all $a \in \mathcal S$, exactly two eigenfunctions vanish on the spherical ellipsoid $\mathcal C_a \cap \mathbb S^d$.

    This should be compared with a result of Bourgain and Rudnick, who showed that a hypersurface in $\mathbb T^d = \mathbb R^d / \mathbb Z^d$ with nowhere vanishing Gauss-Kronecker curvature is contained in the nodal set of at most finitely many eigenfunctions \cite[Theorem 1.2]{BourgainRudnick2012}. The corresponding problem on spheres seems more intricate, as is indicated by the fact that on $\mathbb S^3$ infinitely many eigenfunctions vanish on any ``sphere of latitude'' whose angle to the equator is a rational multiple of $\pi$ (this follows from \cite[Example 4.3]{LogunovMalinnikova2016}).
\end{remark}

If one replaces the Laplace operator by a general elliptic differential operator of second order with smooth coefficients, the ``good'' cones from \Cref{thm:main} still play an important role:

\begin{theorem}
    \label{thm:general}
    Let $L = \sum_{j,k=1}^d a_{jk}(x) \frac{\partial^2}{\partial x_j \partial x_k}+ \sum_{j=1}^d b_{j}(x) \frac{\partial}{\partial x_j} + c(x)$ be an elliptic partial differential operator with smooth coefficients. Let $u$ be a germ at $0$ of a solution to the equation $Lu = 0$ for which $u(0) = 0$ and $\nabla u(0) = 0$, but $\mathrm{Hess}_u(0)$ is nondegenerate.

    Let $\lambda_1\leq\ldots\leq\lambda_d$ denote the eigenvalues of the bilinear form $\mathrm{Hess}_u(0)$ with respect to the bilinear form $a(0)$. If $\left(\frac1\lambda_1,\ldots,\frac1\lambda_d\right)$ is contained in the set $\mathcal S$ introduced in \Cref{thm:main}, then the space of germs at $0$ of solutions to $L$ which vanish on $u^{-1}(\{0\})$ is at most two-dimensional.
\end{theorem}

The structure of this paper is as follows: \Cref{sec:preliminaries} reduces the problem to one about homogeneous harmonic polynomials through  a series of simple and well-known observations. \Cref{sec:niven_equations} contains a discussion of the geometric picture behind the proof of \Cref{thm:main}, and introduces the main tool used to solve this problem, a theorem by Agranovsky and Krasnov relating quadratic divisors of homogeneous harmonic polynomials to the so-called \emph{Niven equations} \cite[Theorem 1.2]{AgranovskyKrasnov2000}. \Cref{sec:proof} establishes \Cref{thm:main} by analyzing the parametric behaviour of solutions to Niven's equations. The proof uses an approximation formula for associated Legendre polynomials due to Landauer \cite{Landauer1951}. \Cref{sec:WKB} is dedicated to establishing the necessary error bounds for this approximation, which turn out to be quite convenient. \Cref{sec:smooth} contains the proof of \Cref{thm:general}.

%% file: Preliminary_Observations.tex
Given a set $\mathcal C \subseteq \mathbb R^d$, let $\mathscr V(\mathcal C)$ denote the space of harmonic polynomials vanishing on $\mathcal C$, and for $N \geq 0$, let $\mathscr V_N(\mathcal C)$ denote the subspace of $\mathscr V(\mathcal C)$ consisting of harmonic homogeneous polynomials of degree $N$.

The following observation, due to D.H. Armitage \cite[Lemma 1]{Armitage1992}, is elementary, but extremely useful: If $\mathcal C$ is a \emph{cone}, i.e. $t \, \mathcal C = \mathcal C$ for all $t > 0$, then $\dim \mathscr V(\mathcal C) = \sum_{N=0}^\infty \dim \mathscr V_N(\mathcal C)$. Of course, $\dim \mathscr V(\mathcal C) \geq \sum_{N=0}^\infty \dim \mathscr V_N(\mathcal C)$. On the other hand, we can write $u \in \mathscr V(\mathcal C)$ as a sum of harmonic homogeneous polynomials, $u = \sum_{N=0}^\infty u_N$. An inductive argument, starting at the first nonvanishing term of this sum, then shows $u_N \in \mathscr V_N(\mathcal C)$ for all $N \geq 0$.

One can furthermore consider the spaces $\mathscr V^\infty(\mathcal C)$, $\mathscr V^R(\mathcal C)$ and $\mathscr V^0(\mathcal C)$ of entire harmonic functions, harmonic functions defined on $\mathbb B_R(0)$, and germs of harmonic functions at $0$, respectively, which vanish on $\mathcal C$ whenever they are defined. Clearly, $\mathscr V(\mathcal C)$ injects into all these spaces, and these maps are surjective if and only if $\dim \mathscr V(\mathcal C) < \infty$.

Restricting our attention to the quadratic cones $\mathcal C_a$ defined above, we observe that $\mathscr V_2(\mathcal C_a) \geq 1$ and $\mathscr V_{d+2}(\mathcal C_a) \geq 1$, since the homogeneous harmonic polynomials
$\frac{x_1^2}{a_1} + \ldots + \frac{x_d^2}{a_d}$ and $x_1\ldots x_n(  \frac{x_1^2}{a_1} + \ldots + \frac{x_d^2}{a_d} )$ always vanish on $\mathcal C_a$. \Cref{thm:main} will follow once we establish that for each $N \in \mathbb N$, $\mathscr V_{N}(\mathcal C_a)$ attains its minimum on a Zariski open subset of $\mathcal A^d$, and that $\min_{a \in \mathcal A^d} \mathscr V_{N}(\mathcal C_a) = 0$ unless $N \in \{2,2+d\}$, where $\min_{a \in \mathcal A^d} \mathscr V_{N}(\mathcal C_a) = 1$.

As noted by Logunov--Malinnikova \cite[Lemma 2.1]{LogunovMalinnikova2016} (synthesizing Theorem 2 and Lemma 4 of \cite{Murdoch1964}), any harmonic polynomial vanishing on the zero set of a given homogeneous harmonic polynomial is in fact divisible by the latter. A straightforward generalization of \cite[Proposition 2.2]{AgranovskyKrasnov2000} then shows that sublevel sets of $\dim \mathscr V_{N}(\mathcal C_a)$ are Zariski-open subsets of $\mathcal A^d$, establishing the first of the two claims. Indeed, since $\dim \mathscr V_{N}(\mathcal C_a)$ is the corank of the linear map $r(x) \mapsto \Delta \left(\left(\frac{x_1^2}{a_1}+ \ldots + \frac{x_d^2}{a_d}\right) r(x)\right)$ on the space of homogeneous polynomials of order $N-2$, its superlevel sets are determined by the vanishing of certain minors of the corresponding matrix, i.e.\ by polynomial equations in $\frac{1}{a_1},\ldots,\frac{1}{a_d}$. Multiplying through by the common denominator shows the superlevel sets are Zariski closed in $\mathcal A^d$.

%% file: Conical_Coordinates_and_Brelot_Choquet.tex
The proper starting point of our proof is the characterization of quadratic divisors of harmonic polynomials, as obtained by Agranovsky and Krasnov in \cite{AgranovskyKrasnov2000}. We will need the following, slightly modified, version of \cite[Theorem 1.2]{AgranovskyKrasnov2000}.

\begin{thmx}
    \label{thm:AgranovskyKrasnov2000}
    Consider the quadratic cone $\mathcal C_a$ in $\mathbb R^d$ defined by
    \begin{align*}
        \frac{x_1^2}{a_1} + \frac{x_2^2}{a_2} + \ldots + \frac{x_d^2}{a_d} = 0,
    \end{align*}
    where $a_1 < a_2 < \ldots < a_d$ are distinct, nonzero real numbers.
    
    Suppose a homogeneous harmonic polynomial of degree $N$ vanishes on $\mathcal C_a$. Then there exist $n \in \mathbb N$ and $\epsilon \in \{0,1\}^d$ with $2n+|\epsilon| = N$ such that
    \begin{align}
    \label{eq:Niven}
        &\sum_{j=1}^d \frac{1+2\epsilon_j}{\xi_k-a_j} + \sum_{\ell\neq k} \frac{4}{\xi_k-\xi_\ell} = 0, &k=1,\ldots,n
    \end{align}
    has a solution $(\xi_k)_{k=1}^n$ with $\xi_k = 0$ for some $k \in \{1,\ldots,n\}$. 
\end{thmx}

In fact, $\dim \mathscr V_N(\mathcal C_a)$ is given by the number of solutions to (\ref{eq:Niven}) with some $\xi_k = 0$, as $n$ and $\epsilon$ range over all combinations where $2n+|\epsilon| = N$.

This slight extension of \cite[Theorem 1.2]{AgranovskyKrasnov2000} shifts the attention from classifying quadratic divisors of harmonic polynomials to the harmonic polynomials themselves, and follows directly from the arguments in \cite{AgranovskyKrasnov2000}. Both to fully justify this reformulation and to provide the geometric picture behind the rather technical calculations of \Cref{sec:proof}, we will provide a short discussion of the setup, culminating in an alternative argument for \Cref{thm:AgranovskyKrasnov2000}.

\subsection{Niven's equations}

It was observed by W.\ D.\ Niven in 1891 \cite{Niven1891} that the product of linear factors and confocal quadratic forms
\begin{align}
    \label{eq:conicHarmonics}
    Q(x) = x_1^{\epsilon_1}x_2^{\epsilon_2}\ldots x_d^{\epsilon_d} \prod_{k=1}^n \left(\frac{x_1^2}{\xi_k - a_1} + \frac{x_2^2}{\xi_k - a_2} + \ldots + \frac{x_d^2}{\xi_k - a_d} \right)
\end{align}
is harmonic if and only if $(\xi_k)_{k=1}^n$ satisfies (\ref{eq:Niven}), where $a_1<a_2<\ldots<a_d$ are real numbers, $n \in \mathbb N$ and $(\epsilon_j)_{j=1}^d \in \{0,1\}^d$.

Indeed, if we denote $K_\xi(x) = \frac{x_1^2}{\xi - a_1} + \frac{x_2^2}{\xi - a_2} + \ldots + \frac{x_d^2}{\xi - a_d}$, an application of the product rule and a partial fraction decomposition show that
\begin{align}
    \label{eq:deltaQ}
    \Delta Q(x) = Q(x) \sum_{k=1}^n \left(\sum_{j=1}^d \frac{1+2\epsilon_j}{\xi_k-a_j} + \sum_{\ell\neq k} \frac{4}{\xi_k-\xi_\ell}\right)\frac{1}{K_{\xi_k}(x)}.
\end{align}
The rational functions $\frac{1}{K_{\xi_k}}$ attain their singularities on distinct quad\-rics, and so are linearly independent. Therefore, $\Delta Q(x) = 0$ if and only if the coefficients of $\frac{1}{K_{\xi_k}}$ in (\ref{eq:deltaQ}) vanish for all $k \in \{1,\ldots,n\}$, i.e. if (\ref{eq:Niven}) is satisfied.

\subsubsection{Number of solutions and orthogonality}
\label{subsec:numberOfSolutions}

The solutions to (\ref{eq:Niven}) are precisely the critical points of the function 
\begin{align*}
& \varphi(\xi) := - \sum_{k=1}^n \sum_{j=1}^d (1+2\epsilon_j)\log|\xi_k-a_j| - 4 \sum_{k=1}^n\sum_{\ell\neq k} \log|\xi_k-\xi_\ell|.
\end{align*}
This function is defined on the complement of the hyperplanes $\{\xi_k = \xi_\ell\}$, $k \neq \ell$, and $\{\xi_k = a_j\}$, $1 \leq k \leq n$, $1 \leq j \leq d$, which cut $\mathbb R^n$ into a finite number of convex regions, on each of which $\varphi$ is convex. Thus, $\varphi$ attains at most one critical point on each of these regions --- namely it's minimum, if it is achieved. Let $m \in \mathbb N^{d-1}$ be a multi-index with $|m| = m_1 + \ldots + m_{d-1} = n$. To it we associate the domain $D_m \subseteq \mathbb R^n$ defined by the linear constraints
\begin{align*}
 a_j < \xi_{k} < a_{j+1}&, && 1 \leq j < d, \ \textstyle{\sum_{i=1}^{j-1}m_i < k \leq \sum_{i=1}^{j}m_i}, \\
 \xi_{k'} < \xi_{k}&, && 1\leq k' < k \leq n.
\end{align*}
On this domain, $\varphi$ attains its minimum, since it approaches $+\infty$ as $\xi$ approaches $\partial D_m$. Up to permutation of the entries of $\xi$, all bounded components of the domain of definition of $\varphi$ arise in this way. On the unbounded components, which are characterized by having at least one entry of $\xi$ outside the interval $[a_1,a_d]$, the function $\varphi$ is unbounded below, and so cannot attain its minimum. Up to permutation of their entries, solutions to Niven's equations (\ref{eq:Niven}) are thus enumerated by multi-indices $m \in \mathbb N^{d-1}$ with $|m| = n$. 

For fixed $a_1 < \ldots < a_d$, the homogeneous harmonic polynomials of degree $N \in \mathbb N$ of the form (\ref{eq:conicHarmonics}) are therefore in one-to-one correspondence with choices of $\epsilon \in \{0,1\}^d$ and $m \in \mathbb N^{d-1}$ satisfying $2|m| + |\epsilon| = N$. We will denote the corresponding functions of the form (\ref{eq:conicHarmonics}) by $Q_{\epsilon,m}$, or sometimes $Q^a_{\epsilon,m}$ to emphasize the dependence on $a$.

As $\epsilon \in \{0,1\}^d$ and $m \in \mathbb N^{d-1}$ range over all possible choices which satisfy $2|m| + |\epsilon| = N$, the restricted functions $Q_{\epsilon,m}\rvert_{\mathbb S^{d-1}}$ are orthogonal with respect to the standard $L^2$ inner product on $\mathbb S^{d-1}$ \cite[Theorem 3.3]{Volkmer1999}. Thus, by a dimension count, they form an orthogonal basis of the space of spherical harmonics of degree $N$. For a crisp and detailed proof of the orthogonality of $Q_{\epsilon,m}\rvert_{\mathbb S^{d-1}}$, see Volkmer's paper \cite{Volkmer1999}, which also derives the properties above in a different way, hinging on integral calculations in conical coordinates.

\subsubsection{Conical Coordinates}

Fix $d$ real numbers $a_1 < a_2 < \ldots < a_d$. The conical coordinates of a point $x \in (0,\infty)^d$ are given by $r = \left(x_1^2 + \ldots + x_d^2\right)^{\frac12}$ and the $d-1$ solutions of the equation
\begin{align*}
    \frac{x_1^2}{s-a_1} + \ldots + \frac{x_d^2}{s-a_d} = 0,
\end{align*}
listed in ascending order. It is easy verified that these always exist, satisfy $a_1 < s_1 < a_2 < \ldots < a_{d-1} < s_{d-1} < a_d$, and that conical coordinates fashion a diffeomorphism from $(0,\infty)^d$ to the coordinate domain $(0,\infty)\times(a_1,a_2)\times \ldots \times (a_{d-1},a_d)$. 
The coordinate hypersurfaces $\{s_j = \mathrm{const.} \}$ are mutually orthogonal, confocal cones of signature $d-j$.
By evaluating the defining relation
\begin{align*}
    \sum_{j=1}^d x_j^2 \prod_{i\neq j} (s-a_i) = r^2 \prod_{j=1}^{d-1} (s-s_j),
\end{align*}
at $s = a_1,\ldots,a_d$, one obtains formulae for $x_1,\ldots,x_d$:
\begin{align}
    \label{eq:x_in_rs}
    x_j^2 = r^2 \frac{\prod_{i=1}^{d-1}(a_j-s_i)}{\prod_{i\neq j} (a_j-a_i)}
\end{align}
It follows that, in conical coordinates, a homogeneous even polynomial of degree $2n$ pulls back to $r^{2n}P(s)$, where $P(s)$ is a \emph{symmetric} polynomial of degree $(d-1)n$. In particular, $K_\xi(x) = \frac{x_1^2}{\xi - a_1} + \frac{x_2^2}{\xi - a_2} + \ldots + \frac{x_d^2}{\xi - a_d}$ becomes
\begin{align*}
    K_\xi(r,s) = r^2 \frac{\prod_{j=1}^{d-1}(\xi-s_j)}{\prod_{j=1}^d(\xi-a_j)}.
\end{align*}
To see this, fix $x \in \mathbb R^d$ and consider the rational function $\xi \mapsto K_\xi(x)$. It is of degree $d$, has poles at $a_1,\ldots,a_d$, and by the construction of the ellipsoidal coordinates has its zeroes precisely at $s_1,\ldots,s_d$. The leading constant $r^2$ is determined by letting $\xi \rightarrow \infty$, as $\lim_{\xi\to\infty} \xi K_\xi(x) = r^2$.

Thus, conical coordinates are well adapted to the functions $Q_{\epsilon,m}$, which take the form
\begin{align*}
    Q_{\epsilon,m}(r,s) = x_1^{\epsilon_1}x_2^{\epsilon_2}\ldots x_d^{\epsilon_d} \, r^{2|m|} C_{\epsilon,m} \prod_{j=1}^{d-1} P_{\epsilon,m}(s_j),
\end{align*}
where $x_j$ should be thought of as shorthand for the algebraic function of $r$ and $s$ given by (\ref{eq:x_in_rs}), $C_{\epsilon,m} \in \mathbb R \setminus \{0\}$, and
$P_{\epsilon,m}$ is the monic polynomial with zeros at the solutions $\xi_1,\ldots,\xi_n$ of Niven's equations (\ref{eq:Niven}) with parameters $m$ and $\epsilon$. To be precise,
\begin{align*}
     C_{\epsilon,m} &= (-1)^{n(d-1)}\prod_{k=1}^n\prod_{j=1}^d (\xi_k - a_j)^{-1}, & P_{\epsilon,m}(\tau) &= \prod_{k=1}^n (\tau - \xi_k).
\end{align*}

\subsubsection{Proof of \Cref{thm:AgranovskyKrasnov2000}}

Following \cite{AgranovskyKrasnov2000}, let us call a function $f$ of $d$ real variables \emph{$\epsilon$-odd}, for $\epsilon \in \{0,1\}^d$, if $f(x_1,\ldots,-x_j,\ldots,x_d) = (-1)^{\epsilon_j} f(x)$.
Fix a quadratic cone $\mathcal C = \{\sum_{j=1}^d x_j^2/a_j = 0\}$, and denote the space of harmonic homogeneous polynomials of degree $N$ vanishing on $\mathcal C$ by $\mathscr V_N(\mathcal C)$. Since $\mathcal C$ is symmetric along coordinate hyperplanes, $\mathscr V_N(\mathcal C)$ splits into a direct sum of spaces of $\epsilon$-odd functions, $\epsilon \in \{0,1\}^d$, which we shall denote by $\mathscr V_{N,\epsilon}(\mathcal C)$.

Now fix $\epsilon \in \{0,1\}^d$, and consider $u \in \mathscr V_{N,\epsilon}(\mathcal C)$, which we can expand as 
\begin{align*}
    u(x) = \hspace{-3mm}\sum_{\substack{m \in \mathbb N^{d-1} \\ 2|m|+|\epsilon| = N}} \hspace{-3mm}\mu_m \, Q_{\epsilon,m}(x),
\end{align*}
since $\{Q_{\epsilon,m}(x)\}_{2|m|+|\epsilon| = N}$ provides a basis of the $\epsilon$-odd homogeneous harmonic polynomials of degree $N$. In conical coordinates, the cone $\mathcal C$ is given by $\{s_{\sigma} = 0\}$, where $\sigma \in \{1,\ldots,d-1\}$ is such that $a_{\sigma} < 0 < a_{\sigma+1}$, and so the condition $u|_{\mathcal C} = 0$ reads
\begin{align*}
    x_1^{\epsilon_1}x_2^{\epsilon_2}\ldots x_d^{\epsilon_d}  \, r^{2|m|} \hspace{-3mm}\sum_{\substack{m \in \mathbb N^{d-1} \\ 2|m|+|\epsilon| = N}} \hspace{-3mm}\mu_m \, C_{\epsilon,m}P_{\epsilon,m}(0) \prod_{j\neq \sigma} P_{\epsilon,m}(s_j) = 0
\end{align*}
Once we show that the polynomials $\prod_{j\neq \sigma} P_{\epsilon,m}(s_j)$ of $d-2$ variables occuring in this expression are all linearly independent, \Cref{thm:AgranovskyKrasnov2000} is established. After all, this would imply that $\mu_m = 0$ unless $P_{\epsilon,m}(0) = 0$, which is the case if and only one of the components of the solution $(\xi_k)_{k=1}^n$ of Niven's equations (\ref{eq:Niven}) with parameters $\epsilon$ and $m$ is zero.

To this end, let $\nu_m$ be coefficients such that $\sum_{m} \nu_m \prod_{j\neq \sigma} P_{\epsilon,m}(s_j) = 0$. For any $t \in (a_d,\infty)$, $P_{\epsilon,m}(t) \neq 0$, which allows us to define
\begin{align*}
    F_t(r,s) := x_1^{\epsilon_1}x_2^{\epsilon_2}\ldots x_d^{\epsilon_d}  \, r^{2|m|} \sum_{m} \nu_m P_{\epsilon,m}(t)^{-1} \prod_{j = 1}^{d-1} P_{\epsilon,m}(s_j).
\end{align*}
By construction, this polynomial vanishes on the hyperplane $\{s_\sigma = t\}$. Since it is also symmetric in $s_1,\ldots,s_{d-1}$, it thus vanishes on $\{s_j=t\}$ for all $j=1,\ldots,n$, and hence is divisible by $r^2 \prod_{j=1}^{d-1} (s_j - t)$. Translating back into Euclidean coordinates, we find that
\begin{align*}
    \sum_{j=1}^d \frac{x_j^2}{t-a_j} \ \Big \vert \ \sum_{m} \nu_m \, C_{\epsilon,m}^{-1} P_{\epsilon,m}(t)^{-1} Q_{\epsilon,m}(x).
\end{align*}
As shown by Brelot and Choquet \cite{BrelotChoquet1955}, no nonnegative polynomial can divide a nonzero harmonic polynomial, 
so $\sum_{m} \nu_m \, C_{\epsilon,m}^{-1} P_{\epsilon,m}(t)^{-1} Q_{\epsilon,m}(x) = 0$. Since the polynomials $Q_{\epsilon,m}$ are linearly independent, $\nu_m = 0$ for all $m$. \qed

%% file: Main_Section.tex
\section{Parametric behaviour of Niven's equations}
\label{sec:proof}

To treat parameters $a \in \mathbb R^d$ where not all components are distinct, it is convenient to introduce the following notation: For a multi-index $m \in \mathbb N^{d-1}$, write $m^\#_j = \sum_{i=1}^{j-1} m_i$, where $j = 1,\ldots,d$.
Given $a_1 \leq \ldots \leq a_d$, $m \in \mathbb N^{d-1}$ and $\epsilon \in \{0,1\}^d$, a convexity argument analogous to that in \Cref{subsec:numberOfSolutions} implies there exists a unique $|m|$-tuple $\xi \in \mathbb R^{|m|}$ such that
\begin{enumerate}
    \item if $a_j = a_{j+1}$, then $a_j = \xi_{m^\#_j + 1} = \ldots = \xi_{m^\#_j + m_j} = a_{j+1}$, while
    \item if $a_j < a_{j+1}$, then $a_j < \xi_{m^\#_j + 1} < \ldots < \xi_{m^\#_j + m_j} < a_{j+1}$, and all such components $\xi_k$ satisfy Niven's equation, 
    \begin{align}
        \label{eq:Niven2}
        \sum_{j=1}^d \frac{1+2\epsilon_j}{\xi_k-a_j} + \sum_{\ell\neq k} \frac{4}{\xi_k-\xi_\ell} = 0.
    \end{align}
\end{enumerate}
Denote this $|m|$-tuple by $\xi_{\epsilon,m}^a$, or just $\xi_{\epsilon,m}$ if the dependence on $a$ is clear.

Suppose $a(t) \in \mathbb R^d$ depends continuously on a parameter $t \in [0,\delta)$, and $a_1(t) < \ldots < a_d(t)$ for all $t \in (0,\delta)$. Write $\xi_{\epsilon,m}(t)$ for $\xi_{\epsilon,m}^{a(t)}$ as introduced above. It is easy to see that $\xi_{\epsilon,m}(t) \to \xi_{\epsilon,m}(0)$ as $t \to 0$, i.e.\ that $\xi_{\epsilon,m}(t)$ is continuous. The next lemma, inspired by Rellich's theorem on eigenvectors of analytic families of self-adjoint operators \cite{Rellich1937}, shows that if the dependence on $t$ is analytic, then $\xi_{\epsilon,m}(t)$ can be continued analytically across $0$.

\begin{lemma}
\label{lem:analytic_continuation}
    Suppose that $a(t) \in \mathbb R^d$ is analytic in $t \in (\alpha,\beta)$, $\alpha < 0 < \beta$, and that for $t \in (0,\beta)$ we have $a_1(t) < \ldots < a_d(t)$. Fix $m \in \mathbb N^d$ and $\epsilon \in \{0,1\}^d$, and consider $\xi_{\epsilon,m}(t) \in \mathbb R^{|m|}$ as introduced above. Then $\xi_{\epsilon,m}: (0,\beta) \to \mathbb R^{|m|}$ extends analytically to a function $\tilde \xi_{\epsilon,m}: (\alpha,\beta) \to \mathbb R^{|m|}$.
\end{lemma}

\begin{proof}
We begin by showing that for real $a \in \mathbb R^d$ with pairwise distinct components, any \emph{complex} solution $\xi \in \mathbb C^N$ of Niven's equations (\ref{eq:Niven}) is in fact real. If $\xi \in \mathbb C^N$ is such a solution, $\bar \xi \in \mathbb C^N$ must also solve (\ref{eq:Niven}), since $a$ is real. Subtracting these equations and dividing by $2i$, we obtain the system
\begin{align}
\label{eq:Imaginary_Part_of_Niven}
    \sum_{j=1}^d \frac{1+2\epsilon_j}{|\xi_k-a_j|^2} \Im \xi_k + \sum_{\ell\neq k} \frac{4}{|\xi_k-\xi_\ell|^2} \Im \xi_k - \sum_{\ell\neq k} \frac{4}{|\xi_k-\xi_\ell|^2} \Im \xi_\ell = 0,
\end{align}
where $k = 1,\ldots,N$. The matrix $A \in \mathbb R^{N \times N}$ given by
\begin{align*}
    A_{k,\ell} = \begin{cases}
        \sum_{j=1}^d \frac{1+2\epsilon_j}{|\xi_k-a_j|^2} + \sum_{\ell'\neq k} \frac{4}{|\xi_k-\xi_{\ell'}|^2}, & k = \ell \\
         - \frac{4}{|\xi_k-\xi_\ell|^2}, & k \neq \ell 
    \end{cases}
\end{align*}
is positive definite, because it is a sum of a diagonal matrix with positive entries and the matrices $\frac{4}{|\xi_k-\xi_\ell|^2} (e_k - e_\ell) (e_k - e_\ell)^T$ (where $e_k \in \mathbb R^N$ denotes the $k^{th}$ unit vector), which are positive semidefinite. Therefore 
(\ref{eq:Imaginary_Part_of_Niven}) implies $\Im \xi = 0$, i.e.\ $\xi \in \mathbb R^N$.

Consider a small disk $\mathbb D_{\delta}(0) \subseteq \mathbb C$, with $\delta \in (0,\beta]$, such that $a(t) \in \mathbb C^d$ has pairwise distinct components for all $t \in \mathbb D_{\delta}(0) \setminus \{0\}$. After multiplying through by the denominators in (\ref{eq:Niven}), we obtain an analytic system of equations in $\mathbb D_{\delta}(0) \times \mathbb C^N$. Its solutions form an analytic set in $D_{\delta}(0) \times \mathbb C^N$, which may have unwanted irreducible components contained in the zero set of a denominator, i.e. in one of the planes $\{\xi_k = a_j \}$ or $\{\xi_k = \xi_\ell \}$, $k \neq \ell$. However, the local irreducible component $\mathscr X$ containing the smooth real curve $(t,\xi_{\epsilon,m}(t))$, $t \in (0,\delta)$, is one-dimensional. After all, near all $t \in (0,\delta)$ it is a smooth complex curve, i.e.\ one-dimensional, and irreducible analytic sets are pure-dimensional (cf. \cite[Chapter 5, \S 4]{GrauertRemmert1984}). 

After possibly shrinking $\delta$ to avoid any singularities of $\mathscr X$ other than the one at $t = 0$, we can thus analytically continue $\xi_{\epsilon,m}(t)$ to a bounded holomorphic curve over the open half-disk $\mathbb D_{\delta}(0) \cap \{t \in \mathbb C: \Im z > 0\}$. However, since (\ref{eq:Niven}) admits no purely complex solutions for $a(t) \in \mathbb R^d$, the boundary value of $\xi_{\epsilon,m}(t)$ on $(-\delta,\delta)$ must be real. Thus, Schwarz' reflection principle implies $\xi_{\epsilon,m}(t)$ in fact extends to the whole disk $\mathbb D_{\delta}(0)$, and so $\xi_{\epsilon,m}:(0,\beta) \to \mathbb R^n$ extends analytically to a function $\tilde\xi_{\epsilon,m}:(-\delta,\beta) \to \mathbb R^n$.

Note that $\tilde\xi_{\epsilon,m}(t)$ still \emph{solves} Niven's equations for all $t \in (-\delta,0)$, although its components are no longer necessarily ordered. The preceding argument however only required $\xi(t)$ to be a solution of Niven's equations, and so may be repeated at all $t \in (\alpha,\beta)$ where $a(t)$ does not have pairwise distinct components (a discrete subset of $(\alpha,\beta)$ since $a(t)$ is analytic in $t$). This yields an analytic continuation $\tilde \xi_{\epsilon,m}(t)$ of $\xi_{\epsilon,m}(t)$ to the whole interval $(\alpha,\beta)$.
\end{proof}

\subsection{Proof of \Cref{thm:main}}
From now on, we assume $d\geq 3$, and pick up the argument begun in \Cref{sec:preliminaries}. Recall that the dimension of the space $\mathscr V_N(\mathcal C_a)$ of harmonic homogeneous polynomials of a fixed degree $N$ which vanish on $\mathcal C_a$ achieves its minimum on a Zariski open subset of $\mathcal A^d$. It remains to show $\min_{a \in \mathcal A^d} \dim \mathscr V_N(\mathcal C_a) = 0$ unless $N \in \{2,2+d\}$, in which case we must show $\min_{a \in \mathcal A^d} \dim \mathscr V_N(\mathcal C_a) = 1$. This will be achieved in five steps.

\begin{enumerate}
    \item According to \Cref{thm:AgranovskyKrasnov2000}, the number of linearly independent harmonic functions vanishing on $\mathcal C_a$ for a fixed $a \in \mathcal A^d$ is given by the number of solutions of Niven's equations with a zero component. As discussed in \Cref{sec:niven_equations}, the solutions of Niven's equations are enumerated by $\epsilon \in \{0,1\}^d$ and $m \in \mathbb N^{d-1}$. On the Zariski open subset of $\mathcal A^d$ where all components of $a \in \mathcal A^d$ are distinct, each solution $\xi_{\epsilon,m}^a$ depends analytically on $a$. It follows that, at a generic point $a \in \mathcal A^d$, the only values of $\epsilon,m$ for which a component of $\xi_{\epsilon,m}^a$ vanishes are those for which this component vanishes identically.
    \item Next, we specialize to three dimensions. We consider the analytic curve $a(t) = (-1,\frac{2}{1+t},\frac{2}{1-t}), t \in [0,1)$, in $\mathcal A^3$ which connects the round cone $\{2x_1^2-x_2^2-x_3^2 = 0\}$ to the ``degenerate cone'' $\{x_1^2-x_2^2 = 0\}$. With the goal of determining all $k_0,\epsilon$ and $m$ such that $\big(\xi_{\epsilon,m}^{a(t)}\big)_{k_0} \equiv 0$, we obtain asymptotic expressions for $\xi_{\epsilon,m}^{a(t)}$ as $t \to 0$ and as $t \to 1$.
    \item We first use the conditions $\big(\xi_{\epsilon,m}^{a(1)}\big)_{k_0} = 0$ and $\frac{d}{dt}\rvert_{t=1} \big(\xi_{\epsilon,m}^{a(t)}\big)_{k_0} = 0$ to obtain a Pell-type equation which drastically restricts the possible values of $k_0$, $m$ and $\epsilon$.
    \item Finally, we show that for these values of $k_0$, $m$ and $\epsilon$, $\big(\xi_{\epsilon,m}^{a(0)}\big)_{k_0} \neq 0$, except in two special cases which correspond to the different behaviour for $N = 2$ and $N=d+2$. This step uses an asymptotic expression for zeroes of associated Legendre polynomials, whose proof is deferred to \Cref{sec:WKB}.
    \item In the last step of the proof, we establish \Cref{thm:main} via an induction on the dimension.
\end{enumerate}

\begin{figure}[ht]
\centering
\includegraphics[width=0.9\textwidth]{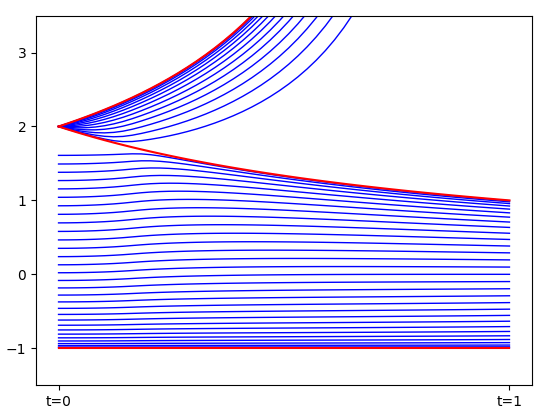}
    \caption{Plot of $\xi_{\epsilon,m}^{a(t)}$ with $a(t)=(-1,\frac{2}{2+t},\frac{2}{2-t})$,
    $m=(31,13)$ and $\epsilon=(1,1,0)$, $t \in [0,1)$}
    \label{fig:31-13_all}
\end{figure}

Steps (2), (3) and (4) are illustrated well by \Cref{fig:31-13_all}, which plots $\xi_{\epsilon,m}^{a(t)}$ for $t \in [0,1)$, $a(t)$ as defined in Step (2), $\epsilon = (1,1,0)$ and $m = (31,13)$. As we will derive in Step (3), these are in fact the next smallest values of $\epsilon$ and $m$ for which $\big(\xi_{\epsilon,m}^{a(1)}\big)_{k_0} = 0$ and $\frac{d}{dt}\rvert_{t=1} \big(\xi_{\epsilon,m}^{a(t)}\big)_{k_0} = 0$, after the ``trivial'' solutions $\epsilon = (0,0,0)$ and $m = (1,0)$ and $\epsilon = (1,1,1)$ and $m = (1,0)$. Step (4) is concerned with establishing that $\big(\xi_{\epsilon,m}^{a(0)}\big)_{k_0} \neq 0$ for all such $\epsilon$ and $m$. As one can see in the figure, the true value of $\big(\xi_{\epsilon,m}^{a(0)}\big)_{k_0}$ is slightly below zero. This is highlighted in \Cref{fig:31-13_lower}.

\begin{figure}[t]
\centering
\includegraphics[width=0.9\textwidth]{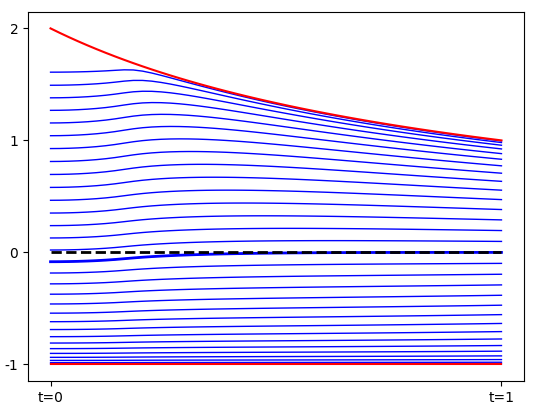}
    \caption{Lower components of $\xi_{\epsilon,m}^{a(t)}$ as in \Cref{fig:31-13_all}, highlighting $\big(\xi_{\epsilon,m}^{a(t)}\big)_{k_0}$ (bold) and the line $\xi=0$ (bold, dashed).}
    \label{fig:31-13_lower}
\end{figure}

\subsubsection{Step 1: Further reduction of the problem}

Let $\widetilde{\mathcal A}^d$ denote the set of $a \in \mathcal A^d$ (as defined in \Cref{sec:preliminaries}) with pairwise distinct components, and let $\widetilde{\mathcal A}^d_0$ be an arbitrary one of its connected components. For $a \in \widetilde{\mathcal A}^d$, \Cref{thm:AgranovskyKrasnov2000} says that $\dim \mathscr V_N(\mathcal C_a)$ is given by the number of $\epsilon \in \{0,1\}^d$, $m \in \mathbb N^{d-1}$ with $|\epsilon| + 2|m| = N$ such that one of the components of $\xi_{\epsilon,m}^a$ vanishes. For fixed $\epsilon$ and $m$, the set of $a \in \widetilde{\mathcal A}^d_0$ such that $\xi_{\epsilon,m}^a$ has a zero component is either of measure zero or all of $\widetilde{\mathcal A}^d_0$, since $\xi_{\epsilon,m}^a$ depends analytically on $a \in \widetilde{\mathcal A}^d_0$. In the latter case, as all components of $\xi_{\epsilon,m}^a$ are distinct, it must be the same component that vanishes for all $a \in \widetilde{\mathcal A}^d_0$, i.e. $\big(\xi_{\epsilon,m}^a\big)_{k_0} \equiv 0$ for some $k_0 \in \{1,\ldots,|m|\}$. Thus, $\min_{a \in \mathcal A_0^d} \dim \mathscr V_N(\mathcal C_a)$ is in fact given by the number of $\epsilon,m$ with $|\epsilon|+2|m| = N$ such that there exists $k_0 \in \{1,\ldots,|m|\}$ with $\big(\xi_{\epsilon,m}^a\big)_{k_0} \equiv 0$ for all $a \in \mathcal A_0^d$.
We will show that, unless $d = 2$, this can only occur for $|m| = 1$, $\epsilon = \{0,\ldots,0\}$ and $|m| = 1$, $\epsilon = \{1,\ldots,1\}$. We already showed in \Cref{sec:preliminaries} that $\min_{a \in \mathcal A^d} \dim \mathscr V_N(\mathcal C_a)$ is attained on a Zariski-open subset of $\mathcal A^d$, which in particular intersects $\widetilde{\mathcal A}_0^d$. Therefore, showing the above implies our claim that $\min_{a \in \mathcal A^d} \dim \mathscr V_N(\mathcal C_a) = 0$ unless $N \in \{2,2+d\}$, in which case $\min_{a \in \mathcal A^d} \dim \mathscr V_N(\mathcal C_a) = 1$.

\subsubsection{Three dimensions: Initial observations.}
\label{sec:3d_1}
For the case $d=3$, consider $a(t) = (-1,\frac{2}{1+t},\frac{2}{1-t})$, $t \in [0,1)$. At $t=0$, the corresponding cone $\mathcal C_t := \mathcal C_{a(t)}$ is the round cone $\{2x_1^2 = x_2^2+x_3^2\}$. As $t \to 1$, $\mathcal C_{t}$ degenerates to a union of two orthogonal planes, $\{x_1=x_2\}$ and $\{-x_1=x_2\}$. Fix $m \in \mathbb N^2$ and $\epsilon \in \{0,1\}^3$ and write $\xi(t) = \xi^{a(t)}_{\epsilon,m}$. Our first task is to obtain asymptotic formulas for $\xi_k(t)$ as $t \to 0$ and as $t \to 1$. The underlying idea behind these is the observation that the homogeneous harmonic polynomial
\begin{align}
    \label{eq:conicHarmonics3D}
    &Q_{\epsilon,m}^{a(t)}(x) = x_1^{\epsilon_1}x_2^{\epsilon_2}x_3^{\epsilon_3} \prod_{k=1}^{|m| } \sum_{j=1}^{3} \frac{x_j^2}{\xi_k(t) - a_j(t)}
\end{align}
converges to a standard spherical harmonic as $t \to 0$ and as $t \to 1$, when properly normalized. This observation was already made by Felix Klein in 1881 \cite[\S 3]{Klein1881}. Since Klein appeals to the reader's geometric intution (\emph{``Geometrische Anschauung''}) for proof, we will provide a rigorous argument, via Niven's equations, as part of the proof of the following lemma. In its statement, the notation $P^\nu_\mu$ stands for the associated Legendre polynomial of degree $\mu$ with angular parameter $\nu$, as defined e.g. in \cite{AbramowitzStegun1964}. Note that, for integers $0 \leq \nu \leq \mu$, $P^\nu_\mu$ has exactly $\mu-\nu$ zeros in the interval $(-1,1)$, which furthermore lie symmetric around $0$. Hence, $\left \lfloor \frac{\mu-\nu}{2} \right\rfloor $ zeros lie in $(0,1)$.

\begin{lemma}
\label{lem:asymptotics_at_0_and_1}
    Consider $\xi(t)$ as introduced above.
    As $t \to 0$,
    \begin{align}
        \label{eq:asymptotics_0_m1}
        \xi_k(t) &= -1 + 3\zeta_k^2 + \mathcal O(|t|), & & 1 \leq k \leq m_1, \\
        \label{eq:asymptotics_0_m2}
        \xi_k(t) &= 2 - 2t\cos\left(\tfrac{2(k-m_1)-1+\epsilon_2}{2m_2+\epsilon_2+\epsilon_3} \pi \right) + \mathcal O(|t|^2), & & m_1 + 1 \leq k \leq |m|,
    \end{align}
    where $\zeta_k$ denotes the $k^{th}$ zero of $ P^{2 m_2+\epsilon_2+\epsilon_3}_{2|m|+|\epsilon|}$ in $(0,1)$. As $t \to 1$,
    \begin{align}
        \label{eq:asymptotics_1_m1}
        \xi_k(t) &= -\cos\left(\tfrac{2k-1+\epsilon_1}{2m_1+\epsilon_1+\epsilon_2} \pi\right) + \mathcal O(|1-t|) & & 1 \leq k \leq m_1 \\
        \label{eq:asymptotics_1_m2}
        \xi_k (t) &= \frac{2}{1-t}\left(1 - \tilde \zeta_{|m|-k+1}^2\right) + \mathcal O(1), & & m_1 + 1 \leq k \leq |m|,
    \end{align}
    where $\tilde \zeta_\ell$ denotes the $\ell^{th}$ zero of $P^{2m_1 + \epsilon_1+\epsilon_2}_{|m|+|\epsilon|}$ in $(0,1)$.
\end{lemma}

\begin{proof}
    We first show, using the scaling and translation invariance of Niven's equations, that there exist $\alpha, \beta \in \mathbb R^{m_1}$ and $\alpha',\beta' \in \mathbb R^{m_2}$ such that
    \begin{align}
        \label{eq:undetermined_asymptotics_0_m1}
        \xi_k(t) &= \alpha_k + \mathcal O(|t|), & & 1 \leq k \leq m_1, \\
        \label{eq:undetermined_asymptotics_0_m2}
        \xi_k(t) &= 2 + 2t \alpha_k' + \mathcal O(|t|^2), & & 1 \leq k \leq m_2
    \end{align}
    as $t \to 0$ and
    \begin{align}
        \label{eq:undetermined_asymptotics_1_m1}
        \xi_k(t) &= \beta_k + \mathcal O(|1-t|) & & 1 \leq k \leq m_1, \\
        \label{eq:undetermined_asymptotics_1_m2}
        \xi_k (t) &= \frac{2}{1-t}\beta_k' + \mathcal O(1), & & 1 \leq k \leq m_2
    \end{align}
    as $t \to 1$. We then use these expressions to normalize $Q_{\epsilon,m}^{a(t)}$ so it has a nonzero limit as $t \to 0$ and $t \to 1$. Finally, we identify both limits as ``standard'' spherical harmonics, defined via associated Legendre polynomials, and obtain the desired asymptotics by comparing the linear and quadratic factors which occur.

    Since $a(t)$ is analytic at $t=0$, \Cref{lem:analytic_continuation} implies that $\xi(t)$ extends analytically across $0$ as well, immediately establishing \eqref{eq:undetermined_asymptotics_0_m1}. The components $\xi_{m_1+1}(t),\ldots,\xi_{|m|}(t)$ are constrained in the interval $(a_2(t),a_3(t))$, and $a_2(0) = a_3(0) = 2$. Thus, $\lim_{t \to 0} \xi_{k}(t) = 2$ for all $k \in \{m_1 + 1,\ldots,m_2\}$, which establishes \eqref{eq:undetermined_asymptotics_0_m2} since $\xi(t)$ is differentiable.

    To obtain \eqref{eq:undetermined_asymptotics_1_m1} and \eqref{eq:undetermined_asymptotics_1_m2}, introduce 
    \begin{align*}
        \tilde a(t) = \frac{1-t}{2} a(t) = \left( \frac{t-1}{2}, \frac{1-t}{1+t}, 1\right),
    \end{align*}
    which extends analytically across $t=1$. The scaling invariance of Niven's equations implies $\xi(t) = \frac{2}{1-t} \xi_{\epsilon,m}^{\tilde a(t)}$. By \Cref{lem:analytic_continuation}, $\xi_{\epsilon,m}^{\tilde a(t)}$ extends analytically across $t=1$ since $\tilde a(t)$ does. Because $\tilde a_1(1) = \tilde a_2(1) = 0$, we have $\big( \xi_{\epsilon,m}^{\tilde a(t)} \big)_k = 0$ for $k = 1,\ldots,m_1$, so $\xi_k(t)$ in fact stays bounded for these values of $k$. This establishes \eqref{eq:undetermined_asymptotics_1_m1} and \eqref{eq:undetermined_asymptotics_1_m2}.

    We now normalize $Q_{\epsilon,m}^{a(t)}(x)$ as follows:
    \begin{align*}
        \tilde Q_{\epsilon,m}^{a(t)}(x) :&= (a_3(t)-a_2(t))^{m_2} Q_{\epsilon,m}^{a(t)}(x) \\ = 
        &\ x_1^{\epsilon_1}x_2^{\epsilon_2}x_3^{\epsilon_3}
        \left( \prod_{k=1}^{m_1} \sum_{j=1}^3 \frac{x_j^2}{\xi_k(t)-a_j(t)} \right)\left( \prod_{k=m_1+1}^{|m|} \sum_{j=1}^3 \frac{x_j^2}{\frac{\xi_k(t)-a_j(t)}{a_3(t)-a_2(t)}} \right).
    \end{align*}
    Note that $a_3(t) - a_2(t) = 4t + \mathcal O(|t|)$ as $t \to 0$ and $a_3(t) - a_2(t) = \frac{2}{1-t} + \mathcal O(1)$ as $t \to 1$. Combining this with \eqref{eq:undetermined_asymptotics_0_m1}, \eqref{eq:undetermined_asymptotics_0_m2}, \eqref{eq:undetermined_asymptotics_1_m1} and \eqref{eq:undetermined_asymptotics_1_m2} yields
    \begin{align}
        \label{eq:spherical_harmonic_0_undetermined}
        \begin{split}
            \lim_{t \to 0} \tilde Q_{\epsilon,m}^{a(t)}(x) = 
            & \ x_1^{\epsilon_1}
            \prod_{k=1}^{m_1} \left( \frac{x_1^2}{\alpha_k + 1} + \frac{x_2^2+x_3^2}{ \alpha_k - 2} \right) \\ 
            \cdot & \, x_2^{\epsilon_2}x_3^{\epsilon_3}
            \prod_{k=m_1+1}^{|m|} \left( \frac{2 x_2^2}{\alpha_k' + 1} - \frac{2 x_3^2}{1 - \alpha_k'}\right),
        \end{split}
    \end{align}
    \begin{align}
        \label{eq:spherical_harmonic_1_undetermined}
        \begin{split}
        \lim_{t \to 1} \tilde Q_{\epsilon,m}^{a(t)}(x) = 
        & \ x_1^{\epsilon_1}x_2^{\epsilon_2}
         \prod_{k=1}^{m_1} \left( \frac{x_1^2}{\beta_k + 1} - \frac{x_2^2}{ 1 - \beta_k} \right) \\
        \cdot & \, x_3^{\epsilon_3}
         \prod_{k=m_1+1}^{|m|} \left( \frac{x_1^2 + x_2^2}{\beta_k'} + \frac{x_3^2}{\beta_k' - 1}\right),
        \end{split}
    \end{align}
    assuming that $\alpha_k \not\in \{-1,2\}$, $\alpha'_k \not\in \{-1,1\}$, $\beta_k \not\in \{-1,1\}$ and $\beta_k' \not\in \{0,1\}$. That this is indeed not the case can be seen as follows: Assume, for example, that $\alpha_1' = -1$ (and so $\xi_{m_1+1}(t) = 2 - 2t + \mathcal O(|t|^2) = a_2(t) + \mathcal O(|t|^2)$. Then the corresponding factor $\sum_{j=1}^3 \frac{x_j^2}{\xi_k(t)-a_j(t)}$ should be normalized by multiplying it with $(\xi_k(t)-a_2(t))$ instead of $(a_3(t)-a_2(t))$, so it converges to $x_2^2$. After proceeding analogously with the remaining factors, we consider $\lim_{t \to 0} \tilde Q_{\epsilon,m}(t)$, which is now a nonzero harmonic homogeneous polynomial divisible by $x_2^2$ -- a contradiction against Brelot and Choquet's theorem that a harmonic polynomial cannot have nonnegative divisors.
    
    Returning to \eqref{eq:spherical_harmonic_0_undetermined}, we observe that $\lim_{t \to 0} \tilde Q_{\epsilon,m}^{a(t)}(x)$ is a harmonic homogeneous polynomial of the form $|x|^{2m_1 + \epsilon_1} f(|x|^{-1}x_1) g(x_2,x_3)$, with $f$ a polynomial of degree $2m_1 + \epsilon_1$ and $g$ a product of $2m_2+\epsilon_2+\epsilon_3$ homogeneous linear factors. It is well known, and easy to see e.g.\ by separating variables on the sphere $\{|x|=1\}$, that this means $f$ is a multiple of the associated Legendre polynomial $P^{2m_2 + \epsilon_2 + \epsilon_3}_{2|m|+|\epsilon|}$, and $g(x_2,x_3) = \Re \omega\, (x_2+ix_3)^{2m_2 + \epsilon_2 + \epsilon_3}$ for some $\omega \in \mathbb C^\ast$. Taking into account that $g$ is divisible by $x_2^{\epsilon_2}x_3^{\epsilon_3}$, we obtain
    \begin{align}
        \label{eq:spherical_harmonic_0_determined}
        \begin{split}
            \lim_{t \to 0} \tilde Q_{\epsilon,m}^{a(t)}(x)
        = 
        & \ c_{\epsilon,m}\, |x|^{2m_1 + \epsilon_1} P^{2 m_2+\epsilon_2 + \epsilon_3 }_{2|m|+|\epsilon|}\left(|x|^{-1}x_1\right)  \\ \cdot & \, \Re \left(i^{\epsilon_3}(x_2+i x_3)^{2 m_2+\epsilon_2+\epsilon_3}\right),
        \end{split}
    \end{align}
    for an (irrelevant) constant $c_{\epsilon,m} \in \mathbb R$. Comparing linear and quadratic factors occurring in the expression \eqref{eq:spherical_harmonic_0_determined} with those in \eqref{eq:spherical_harmonic_0_undetermined} yields \eqref{eq:asymptotics_0_m1} and \eqref{eq:asymptotics_0_m2}. Indeed, comparing the zeros of
    \begin{align*}
        |x|^{2m_1 + \epsilon_1}  P_{2|m|+|\epsilon|}^{2m_2 + \epsilon_2 + \epsilon_3}\left(\frac{x_1}{|x|}\right) &= x_1^{\epsilon_1} \prod_{k=1}^{m_1} (x_1^2 - \zeta_k^2 |x|^2) \text{ and } \\
        x_1^{\epsilon_1}
             \prod_{k=1}^{m_1} \left( \frac{x_1^2}{\alpha_k + 1} + \frac{x_2^2+x_3^2}{ \alpha_k - 2} \right) &= x_1^{\epsilon_1}
             \prod_{k=1}^{m_1}  \frac{(\alpha_k+1)|x|^2 - 3 x_1^2}{(\alpha_k + 1)(\alpha_k-2)}
    \end{align*}
    shows $\alpha_k = -1 + 3\zeta_k^2$ (since both $\alpha_k$ and $\zeta_k$ are ordered increasingly). By comparing linear factors in $\Re \left(i^{\epsilon_3}(x_2+i x_3)^{2 m_2+\epsilon_2+\epsilon_3}\right)$, which agrees with
    \begin{align*}
          x_2^{\epsilon_2}x_3^{\epsilon_3}
            \prod_{k=1}^{m_2} \left( \cos\left(\tfrac{2k-1+\epsilon_2}{2m_2+\epsilon_2+\epsilon_3} \tfrac{\pi}{2} \right)^2 \! x_2^2 - \sin\left(\tfrac{2k-1+\epsilon_2}{2m_2+\epsilon_2+\epsilon_3} \tfrac{\pi}{2} \right)^2 \! x_3^2 \right)
    \end{align*}
    up to a constant, and $x_2^{\epsilon_2}x_3^{\epsilon_3} \prod_{k=m_1+1}^{|m|} \left( \frac{2 x_2^2}{\alpha_k' + 1} - \frac{2 x_3^2}{1 - \alpha_k'}\right)$, we find that 
    \begin{align*}
        \frac{\alpha_k'+1}{1-\alpha_k'} = \frac{\sin\left(\tfrac{2k-1+\epsilon_2}{2m_2+\epsilon_2+\epsilon_3} \tfrac{\pi}{2} \right)^2}{\cos\left(\tfrac{2k-1+\epsilon_2}{2m_2+\epsilon_2+\epsilon_3} \tfrac{\pi}{2} \right)^2}.
    \end{align*}
    After some simplification we obtain $\alpha'_k = - \cos\left(\tfrac{2k-1+\epsilon_3}{2m_2+\epsilon_2+\epsilon_3} \pi \right)$ as claimed.

    Arguing analogously in the case of the limit $t \to 1$, we find that
    \begin{align}
        \label{eq:spherical_harmonic_1_determined}
    \begin{split}
        \lim_{t \to 1} \tilde Q_{\epsilon,m}^{a(t)}(x)
        = 
        & \ c'_{\epsilon,m} \, \Re \left(i^{\epsilon_2} (x_1+i x_2)^{2 m_1+\epsilon_1+\epsilon_2} \right) \\ \cdot & \, |x|^{2m_2 + \epsilon_3} P^{2 m_1 + \epsilon_1 + \epsilon_2 }_{2|m|+|\epsilon|}\left(|x|^{-1}x_3\right),
    \end{split}
    \end{align}
    for some $c'_{\epsilon,m} \in \mathbb R$. The comparison of coefficients yielding $\beta$ works exactly as in the case of $\alpha'$, except with $m_2$, $\epsilon_2$ and $\epsilon_3$ replaced by $m_1$, $\epsilon_1$ and $\epsilon_2$. For $\beta'$, we must compare the factors of
    \begin{align*}
        |x|^{2m_2 + \epsilon_3}  P_{2|m|+|\epsilon|}^{2m_1 + \epsilon_1 + \epsilon_2}\left(\frac{x_3}{|x|}\right) &= x_3^{\epsilon_3} \prod_{k=1}^{m_2} (x_3^2 - \tilde \zeta_k^2 |x|^2) \text{ and } \\
        x_3^{\epsilon_3}
             \prod_{k=1}^{m_2} \left( \frac{x_1^2+x_2^2}{\beta_k'} + \frac{x_3^2}{ \beta_k' - 1} \right) &= x_3^{\epsilon_3}
             \prod_{k=1}^{m_2}  \frac{(\beta_k'-1)|x|^2 + x_3^2}{\beta_k'(\beta_k'-1)},
    \end{align*}
    which immediately leads to $\beta'_k = 1-\tilde\zeta_{m_2-k+1}^2$ (since $(1-\tilde\zeta_{k}^2)_{k=1}^{m_2}$ and $(\beta'_k)_{k=1}^{m_2}$ are in opposite order).
\end{proof}

Suppose now that there exists $k_0$ such that $\xi_{k_0}(t) = 0$ for all $t \in [0,1)$. The asymptotics at $t = 1$ from \Cref{lem:asymptotics_at_0_and_1} then imply $\cos(\frac{2k_0-1+\epsilon_1}{2m_1+\epsilon_1+\epsilon_2} \pi) = 0$. This is only possible if $\epsilon_1 = \epsilon_2$, $2 \nmid m_1$ and $k_0 = \frac{m_1+1}{2}$.

On the other hand, $\xi_{k_0}(0) = -1 + 3 \zeta_{k_0}^2 = 0$ means that the $k_0^{th}$ positive zero of $P^{2m_2 + \epsilon_2+\epsilon_3}_{2|m|+\epsilon|}$ falls on $\frac{1}{\sqrt{3}}$. As already mentioned in the introduction, Mangoubi and Weller Weiser \cite{MangoubiWellerWeiser} conjectured that this never occurs for \emph{any} zero of this polynomial, except in three cases which in our setup correspond to $m = (1,0)$, $\epsilon = (0,0,0)$, $m = (1,0)$, $\epsilon = (1,1,1)$ and $m = (1,1)$, $\epsilon = (1,0,0)$. This claim boils down to a difficult number-theoretic problem, which is substantially reduced, but not solved, in \cite{MangoubiWellerWeiser}. A proof of this conjecture, together with the preceding discussion, would immediately imply \Cref{thm:main}.

We will instead use the assumption that $\xi_{k_0}(t) = 0$ for all $t \in [0,1)$ to restrict the range of possible $m$ and $\epsilon$, by taking a derivative at $t=1$, and then show that for these parameters, $\xi_{k_0}(0) = -1 + 3 \zeta_{k_0}^2 \neq 0$, unless $|m|=1$.

\subsubsection{Three dimensions: Endpoint calculations at $t=1$.}
As $t \to 1$, the parameter $a(t)$ becomes unbounded, and so does $\xi(t)$. However, $(1-t)a(t)$ is analytic across $t=1$, hence \Cref{lem:analytic_continuation} and the scaling invariance of Niven's equations imply that $\xi(t)$ merely has a simple pole at $t=1$. Since $\xi_k(t)$ stays bounded for $k = 1,\ldots,m_1$, it must analytically continue across $t=1$.

For $k = m_1 + 1,\ldots,|m|$, let $\tilde \xi_k(t) = (1-t)\xi_k(t)$, which is analytic across $t=1$ as well. Define 
\begin{align}
\label{eq:reduced_Niven}
\begin{split}
    F&: \mathbb R^{m_1} \times [0,1] \to \mathbb R^{m_1}, \\
    F(y)_k &= \frac{1+2\epsilon_1}{y_k + 1} + \frac{1+2\epsilon_2}{y_k - \frac{2}{1+t}}  + \frac{(1-t)(1+2\epsilon_3)}{(1-t)y_k - 2} \\
    &+ \sum_{\ell = m_1 + 1}^{|m|} \frac{4(1-t)}{y_k(1-t) - \tilde\xi_\ell(t)} + \sum_{\ell\neq k} \frac{4}{y_k - y_\ell}.
\end{split}
\end{align}
The point of this definition is that $y(t) := (\xi_k(t))_{k=1}^{m_1}$ solves $F(y(t),t)=0$, and $F$ is analytic on a neighborhood of $(y(1),1) \in \mathbb R^{m_1} \times \mathbb R$, unlike the original system of equations (\ref{eq:Niven}). This allows us to calculate $y'(1)$ by the chain rule, via $y'(1) = -D_y F(y(1),1)^{-1} D_t F(y(1)),1)$. We first compute
\begin{align*}
    D_t F(y,1)_k &= -\frac{1 + 2\epsilon_2}{2(y_k - 1)^2} + \frac{1+2\epsilon_3}{2} + \sum_{\ell = 1}^{m_2} \frac{4}{\tilde \xi_{m_1+\ell}(1)} \\
    D_y F(y,1)_{k\ell} &= \begin{cases}
        - \frac{1 + 2 \epsilon_1}{(y_k + 1)^2} - \frac{1 + 2 \epsilon_2}{(y_k - 1)^2} -  \sum\limits_{\ell\neq k} \frac{4}{(y_k - y_\ell)^2}, & k = \ell \\
        \frac{4}{(y_k - y_\ell)^2}, & k \neq \ell
    \end{cases}
\end{align*}
The matrix $D_y F(y,1)_{k\ell}$ is negative definite by the same argument as in the proof of \Cref{lem:analytic_continuation}, and hence invertible.

 Recall that according to (\ref{eq:asymptotics_1_m2}), $\tilde \xi_{m_1+\ell}(1) = 2 - 2\tilde \zeta^2_{m_2-\ell+1}$ is the $(m_2-\ell+1)^{th}$ positive zero of $P^{2m_1 + \epsilon_1+\epsilon_2}_{2|m|+|\epsilon|}$. Since $P^{2m_1 + \epsilon_1+\epsilon_2}_{2|m|+|\epsilon|}$ has a zero at $0$ if and only if $\epsilon_3$ is odd, and since its zeroes lie symmetric with respect to the origin,
 \begin{align*}
     \frac{\epsilon_3}{2} + \sum_{\ell = 1}^{m_2} \frac{2}{\tilde \xi_{m_1+\ell}(1)} = \sum_{\zeta} \frac{1}{2 - 2 \zeta^2},
 \end{align*}
 where the latter sum ranges over all zeroes of $P^{2m_1 + \epsilon_1+\epsilon_2}_{2|m|+|\epsilon|}$ not equal to $\pm 1$. Let us calculate this sum for an arbitrary associated Legendre polynomial $P^\nu_\mu$, using Rodrigues' formula $P^\nu_\mu(x) = \frac{1}{2^\mu \mu!}(1-x^2)^{\frac{\nu}{2}} \left(\frac{d}{dx}\right)^{\mu+\nu} (x^2-1)^{\mu}$ (see e.g.\ \cite[Equations 8.6.6 and 8.6.18]{AbramowitzStegun1964}). The zeroes of interest are those of the polynomial $p(x) := \left(\frac{d}{dx}\right)^{\mu+\nu} (x^2-1)^{\mu}$. It follows by Vieta's formula that
 \begin{align*}
     \sum_{\zeta} \frac{1}{2-2\zeta^2} = \frac14 \sum_{\zeta} \frac{1}{1+\zeta} + \frac14 \sum_{\zeta} \frac{1}{1-\zeta} = \frac12  \sum_{\zeta} \frac{1}{1-\zeta} = \frac12  \frac{p'(1)}{p(1)},
 \end{align*}
 since the zeros of $p$ lie symmetric around $0$, and $1-\zeta$ ranges over the zeroes of $p(1-x)$. As $p(1) = \binom{\mu+\nu}{\nu} \frac{\mu! \mu!}{(\mu-\nu)!}2^{\mu-\nu}$ and $p'(1) = \binom{\mu+\nu+1}{\nu+1} \frac{\mu! \mu!}{(\mu-\nu-1)!}2^{\mu-\nu-1}$ by the product rule, we obtain $\sum \frac{1}{2-2\zeta^2} = \frac{(\mu-\nu)(\mu+\nu+1)}{4(\nu+1)}$. Hence,
 \begin{align*}
     D_t F(y,1)_k &= -\frac{1 + 2\epsilon_2}{2(y_k - 1)^2} + K, \\ K :\!&= \frac12 + \frac{(2m_2+\epsilon_3)(4m_1+2m_2+2\epsilon_1+2\epsilon_2+\epsilon_3+1)}{2(2m_1+\epsilon_1+\epsilon_2+1)}.
 \end{align*}

The explicit calculation of $y'(1)$ is enabled by two arithmetic miracles, which we shall isolate in a lemma.


 \begin{lemma}
 \label{lem:miracle}
    Let $\delta \in \mathbb R_{>0}$, and suppose $\xi \in \mathbb R^n$ solves
    \begin{align*}
        \frac{\delta}{\xi_k-1} + \frac{\delta}{\xi_k+1} + 4\sum_{\ell \neq k} \frac{1}{\xi_k - \xi_\ell} = 0.
    \end{align*}
    Let $A \in \mathbb R^{n\times n}$ be given by
    \begin{align*}
        A_{k\ell} &= \begin{cases} 4 \sum_{\ell' \neq k} (\xi_k-\xi_{\ell'})^{-2} + \delta (\xi_k-1)^{-2} + \delta (\xi_k+1)^{-2} & k = \ell, \\ -4(\xi_k-\xi_\ell)^{-2} & k\neq \ell, \end{cases}
    \end{align*}
    and $v,w \in \mathbb R^n$ by $v = \left(\frac{1+\xi_k}{2\delta}\right)_{k=1}^n$ and $w = \left(\frac{1-\xi_k^2}{4(n-1)+2\delta}\right)_{k=1}^n$. Then
    \begin{align*}
        Av &= \left( (\xi_k-1)^{-2} \right)_{k=1}^n, & Aw &= \left(1\right)_{k=1}^n.
    \end{align*}
\end{lemma}

\begin{proof}
Both identities follow from direct calculation. In these computations, we first plug in the given expressions for $A$ and $v$ (resp. $w$), and then simplify them using the equation we assumed is satisfied by $\xi$.
\begin{align*}
    (Av)_k &= \frac{1}{2\delta}
    \left( 4 \sum_{\ell \neq k} (\xi_k - \xi_{\ell})^{-1} + \delta(\xi_k + 1)(\xi_k - 1)^{-2} + \delta(\xi_k + 1)^{-1} \right) \\
    & = \frac{1}{2}
    \left( (\xi_k + 1)(\xi_k - 1)^{-2} - (\xi_k - 1)^{-1} \right) = (\xi_k - 1)^{-2}, \\
    (Aw)_k &= \frac{1}{4(n-1)+2\delta}
    \left( 4 \sum_{\ell \neq k} \frac{\xi_\ell + \xi_k}{\xi_\ell - \xi_k} + \delta \frac{1+\xi_k}{1-\xi_k} + \delta \frac{1-\xi_k}{1+\xi_k}\right) \\
    &= \frac{1}{4(n-1)+2\delta}
    \left( 4(n-1) + \sum_{\ell \neq k} \frac{8\xi_k}{\xi_\ell - \xi_k} + 2\delta + \frac{2\xi_k \delta}{1-\xi_k} - \frac{2\xi_k \delta}{1+\xi_k}\right) \\
    &= 1.
\end{align*}
Perhaps surprisingly, the explicit expressions for $\xi_k$ derived in \Cref{sec:3d_1} played no role in our calculations.
\end{proof}

Because the assumption $\xi_{k_0}(1) = 0$ implied $\epsilon_1 = \epsilon_2$ (\Cref{sec:3d_1}), we can apply \Cref{lem:miracle} to obtain
\begin{align*}
    \xi_k'(1) = - \frac{1+\xi_k(1)}{4} + K \frac{1-\xi_k(1)^2}{4m_1 + 4 \epsilon_1 - 2}.
\end{align*}
Plugging the expression for $K$ (and recalling that $\epsilon_1 = \epsilon_2$), as well as our assumption that $\xi_{k_0}(1) = \xi'_{k_0}(1) = 0$, yields the equation
\begin{align}
    \begin{split}
    &(2m_1 +2\epsilon_1 - 1)(2m_1 + 2\epsilon_1 + 1) = 2K(2m_1 + 2\epsilon_1 + 1) \\ = \ &2m_1 + 2\epsilon_1 + 1 + (2m_2 + \epsilon_3)(4m_1 + 2m_2 + 4\epsilon_1 + \epsilon_3 + 1).
    \end{split}
\end{align}
After simplifying, we obtain the Diophantine equation
\begin{align}
\label{eq:pell}
    (2|m|+|\epsilon|)^2 + (2|m|+|\epsilon|) - 8(m_1+\epsilon_1)^2 + 2 = 0.
\end{align}
Its solutions are given by $2|m|+|\epsilon| = \frac{1}{2}(\sqrt{32(m_1+\epsilon_1)^2-7}-1)$, whenever $32(m_1+\epsilon_1)^2-7$ is a square. Since $\xi_{k_0}(1) = 0$ implied $m_1$ was odd, $m_1 + \epsilon_1$ determines both $m_1$ and $\epsilon_1$. The Pell-type equation $p^2 - 32 q^2 = -7$ admits infinitely many solutions, of which the smallest one, $(p,q) = (5,1)$, leads to $m=(1,0)$, $\epsilon=(0,0,0)$. The remaining non-negative solutions of this equation (obtained via \cite[Theorem 3.3]{ConradPellII}) lead to the following formula for $m$ and $\epsilon$, where $\psi = 17+12\sqrt{2}$ and $(\alpha,\sigma)$ ranges over $\mathbb N_{\geq 1} \times \{-1,1\}$:
\begin{align}
\label{eq:solution_formula}
\begin{split}
    m_1 + \epsilon_1 &= \frac{(4\sqrt{2} + \sigma 5)\psi^\alpha + (4\sqrt{2} - \sigma 5)\psi^{-\alpha}}{8\sqrt{2}}\\
    2|m|+|\epsilon| &= \frac{(4\sqrt{2} + \sigma 5)\psi^\alpha - (4\sqrt{2} - \sigma 5)\psi^{-\alpha}-2}{4}
\end{split}
\end{align}
One more ``small'' solution, namely $m=(1,0)$, $\epsilon=(1,1,1)$, arises from $(\alpha,\sigma) = (1,-1)$. All remaining solutions satisfy $m_1 + \epsilon_1 \geq 32$ and $2|m|+|\epsilon| \geq 90$, which will be important in the next step.

Later, we will need an upper bound for the quotient $\frac{m_1+\epsilon_1+1}{4|m|+2|\epsilon|}$ valid for all solutions of \eqref{eq:pell} with $|m| > 1$. For fixed $\sigma \in \{-1,1\}$, $\psi^{-\alpha}(m_1+\epsilon_1+1)$ decreases with $\alpha$, while $\psi^{-\alpha}(4|m|+2|\epsilon|)$ increases. Thus, $\frac{m_1+\epsilon_1+1}{4|m|+2|\epsilon|}$ decreases in $\alpha$ as well, and so is bounded above by the value taken at the smallest solutions with $|m| > 1$ corresponding to $\sigma = 1$ and $\sigma = -1$, respectively. These are $(m_1+\epsilon_1,2|m|+|\epsilon|) = (32,90)$ and $(67,189)$. Hence,
\begin{align}
    \label{eq:upper_bound_for_quotient}
    \frac{m_1+\epsilon_1+1}{4|m|+2|\epsilon|} \leq \max \left( \frac{33}{180}, \frac{68}{378}\right) = \frac{33}{180}.
\end{align}

\subsubsection{Three dimensions: Endpoint calculations at $t=0$.}

It remains to close the argument by showing that the $k_0^{th}$ positive zero of $P^{2m_2+\epsilon_2+\epsilon_3}_{2|m|+|\epsilon|}$, denoted by $\zeta_{k_0}$, does not fall on $\frac1{\sqrt{3}}$ when $\xi_{k_0}(1)=\xi_{k_0}'(1) = 0$. As derived in the previous subsection, the latter conditions imply $\epsilon_1=\epsilon_2$, $k_0 = \frac{m_1+1}{2}$, and that $m$, $\epsilon$ are given by (\ref{eq:solution_formula}), which we shall therefore assume for the remainder of this subsection. After excluding the two trivial solutions with $m=(1,0)$ we may assume $2|m|+|\epsilon| \geq 90$.

We will use an approximation formula (\Cref{cor:zeros}) for zeros of associated Legendre polynomials proven in \Cref{sec:WKB} below.
Let us introduce the auxiliary quantities
\begin{align*}
    h = \frac{1}{\sqrt{(2|m|+|\epsilon|)(2|m|+|\epsilon|+1)}}, \ 
    \nu = \frac{2m_2+\epsilon_2+\epsilon_3}{\sqrt{(2|m|+|\epsilon|)(2|m|+|\epsilon|+1)}}
\end{align*}
and $\eta = \sqrt{1-\nu^2}$. Under our assumptions, $h \leq \frac{1}{\sqrt{8190}} \sim 0.011$.
Equation (\ref{eq:pell}) suggests $\nu \sim 1-\frac{1}{\sqrt{2}}$. In fact, we can use (\ref{eq:pell}) to bound $\nu$ from above unconditionally (recall that $\epsilon_1=\epsilon_2$):
\begin{align*}
    \nu &\leq \frac{2m_2+\epsilon_3+1}{\sqrt{(2|m|+|\epsilon|)(2|m|+|\epsilon|+1)}} = \frac{2|m|+|\epsilon| + 1 - 2 m_1 - \epsilon_1 - \epsilon_2}{\sqrt{(2|m|+|\epsilon|)(2|m|+|\epsilon|+1)}} \\ &< 1 + h - \frac{2 (m_1 + \epsilon_1)}{\sqrt{(2|m|+|\epsilon|)(2|m|+|\epsilon|+1)}} \leq 1 - \frac{1}{\sqrt{2}}+h
\end{align*}
 Thus, $\nu \leq 1 - \frac{1}{\sqrt{2}}+\frac{1}{\sqrt{8190}} < 0.304$. By \Cref{cor:zeros}, proven in \Cref{sec:WKB}, the $k^{th}$ zero $\zeta_k$ of $P^{2m_2+\epsilon_2+\epsilon_3}_{2|m|+|\epsilon|}$ is well approximated by the (unique) solution $x \in (0,1)$ of the equation
\begin{align}
    \label{eq:approximation_formula_in_use}
        \tan^{-1}\left(\tfrac{\frac{x}{\eta}}{\sqrt{1-(\frac{x}{\eta})^2}}\right) - \nu \, \tan^{-1}\left(\nu \, \tfrac{\frac{x}{\eta}}{\sqrt{1-(\frac{x}{\eta})^2}}\right) = h \left(k-\tfrac{1+(-1)^{\epsilon_1}}{4}\right) \pi.
\end{align}
Denote the expression on the left hand side by $\vartheta(x)$. By the error estimate provided in \Cref{cor:zeros}, \Cref{sec:WKB}, $\vartheta(\zeta_k)$ differs from $h \left(k-\tfrac{1+(-1)^{\ell-m}}{4}\right) \pi$ by at most $\frac h2(2m_1 + \epsilon_1 - 2k)^{-1}$. For fixed $x$, $\vartheta(x)$ is decreasing in $\nu$, hence
\begin{align*}
    \vartheta\left(\frac{1}{\sqrt{3}}\right) \geq 0.5818\ldots,
\end{align*}
the value obtained by plugging in the upper bound $0.304$ for $\nu$.
On the other hand, the upper bound \eqref{eq:upper_bound_for_quotient} for $\frac{m_1+\epsilon_1 + 1}{4|m|+2|\epsilon|}$ implies
\begin{align*}
    h \left(k_0-\tfrac{1+(-1)^{\epsilon_1}}{4}\right) \pi < \frac{m_1+\epsilon_1 + 1}{4|m|+2|\epsilon|} \pi \leq \frac{33}{180}\pi = 0.5759\ldots
\end{align*}

The difference between these numbers is around $0.0059$, which is small but significant: It is greater than $\frac h2(2m_1 + \epsilon_1 - 2k_0)^{-1}$, which does not exceed $0.002$. Thus, the $k_0^{th}$ positive zero of $P^{2m_2+\epsilon_2+\epsilon_3}_{2|m|+|\epsilon|}$ does \emph{not} fall on $\frac1{\sqrt{3}}$, but on a slightly smaller number. This completes the proof of \Cref{thm:main} in three dimensions, contingent on the error bound for the approximation formula (\ref{eq:approximation_formula_in_use}), which will be proved in \Cref{sec:WKB}.

\subsubsection{Higher dimensions.}

The higher dimensional case follows by induction on $d$, with the base case $d=3$ established above. Fix a connected component $\widetilde{\mathcal A}^d_0 \subseteq \widetilde{\mathcal A}^d$, and suppose $m,\epsilon$ and $k_0$ are such that $\big(\xi_{\epsilon,m}^a\big)_{k_0} \equiv 0$ on $\widetilde{\mathcal A}^d_0$. Let $j_0 \in {1,\ldots,d-1}$ be the index such that $a_{j_0}< 0 < a_{j_0+1}$ on $\widetilde{\mathcal A}^d_0$. Since $d \geq 4$, we may assume $j_0 \not\in \{1,d-1\}$.

Consider $m' = (m_1,\ldots,m_{d-2})$, $\epsilon'=(\epsilon_1,\ldots,\epsilon_{d-1})$. Suppose it is not the case that $|m'|=1$ and $\epsilon' = \{0,\ldots,0\}$ or $\{1,\ldots,1\}$. The induction hypothesis implies that there exists $a' \in \widetilde {\mathcal A}^{d-1}$, which can be assumed to lie on the boundary of $\widetilde{\mathcal A}^d_0$ in the sense that $a'_{j_0-1} < 0 < a'_{j_0}$, for which no component of $\xi_{m',\epsilon'}^{a'}$ vanishes. Since both $a_{d-1}$ and $a_d$ are positive, we can find an analytic curve $a: [0,1) \to \widetilde{\mathcal A}^d_0$ which has a simple pole at $t=1$ where
\begin{align*}
    \lim_{t \to 1} a_j(t) = \begin{cases}
        a'_j & j \neq d \\
        \infty & j = d
    \end{cases}
\end{align*}
It follows that $\lim_{t \to 1} \big(\xi_{m,\epsilon}^a(t)\big)_k = \xi_{m',\epsilon'}^{a'} \neq 0$ for $k \in \{1,\ldots,m_{d-2}^\#\}$, and $\lim_{t \to 1} \big(\xi_{m,\epsilon}^a(t)\big)_k = \infty$ for $k \in \{m_{d-2}^\#+1, \ldots, |m|\}$, a contradiction. Thus, $|m'| = 1$ and $\epsilon'=(\epsilon_1,\ldots,\epsilon_{d-1})$. In particular, $m_{j_0} = 1$ and $m_{j} = 0$ for all $j\neq j_0$.

Completely analogously, we can also send $a_1$ to $-\infty$ and thereby establish that $|(m_2,\ldots,m_{d-1})|=1$ and $(\epsilon_2,\ldots,\epsilon_d) = \{0,\ldots,0\}$ or $\{1,\ldots,1\}$. This completes the proof of \Cref{thm:main}. \qed

%% file: WKB.tex
The proof of \Cref{thm:main} required some knowledge of the approximate location of specific nontrivial zeros of associated Legendre functions, together with precise and \emph{absolute} error bounds, the latter in order to avoid missing counterexamples of small degree. In this section, we derive an approximation formula for associated Legendre polynomials, which seems to have been first discovered by Landauer in 1951 \cite{Landauer1951}. The derivation presented here differs from that in \cite{Landauer1951} and will come with explicit error bounds (\Cref{cor:better_error} and \Cref{cor:zeros}), which are absent in \cite{Landauer1951}. These bounds are valid for all parameters, contain no implicit constants and have a very simple form.

Let us introduce normalized associated Legendre polynomials $\overline{P_\ell^m}(x)$ as follows: If $\ell - m$ is even, let $\overline{P_\ell^m}(x) = P_\ell^m(x)/P_\ell^m(0)$, while for $\ell - m$ odd, $\overline{P_\ell^m}(x) = \sqrt{l(l+1)-m^2} \, P_\ell^m(x)/(P_\ell^m)'(0)$. This normalization has the effect that the local extrema of $\overline{P_\ell^m}(x)$ nearest to the origin take values close to $\pm 1$. In fact, $\overline{P_\ell^m}(x)$ is almost perfectly enveloped by $\pm \sqrt[4]{\frac{\ell(\ell+1)-m^2}{\ell(\ell+1)(1-x^2)-m^2}}$, as we will show in \Cref{prop:WKB_approximation} below.

\begin{proposition}
\label{prop:WKB_approximation}
Let $0 \leq m \leq \ell$ be integers, and introduce $h = \frac{1}{\sqrt{\ell(\ell+1)}}$, $\nu = \frac{m}{\sqrt{l(l+1)}}$, $\eta = \sqrt{1-\nu^2}$. Let
\begin{align}
    \label{eq:phase}
    \vartheta(x) = \tan^{-1}\left(\frac{\frac{x}{\eta}}{\sqrt{1-(\frac{x}{\eta})^2}}\right) - \nu \, \tan^{-1}\left(\nu \, \frac{\frac{x}{\eta}}{\sqrt{1-(\frac{x}{\eta})^2}}\right).
\end{align}On the interval $(-\eta,\eta)$, the normalized associated Legendre function $\overline{P_\ell^m}(x)$ is approximated well by 
    \begin{align}
    \label{eq:approximation}
    W^m_\ell(x) = \begin{cases}
        \left(1-(\tfrac{x}{\eta})^2\right)^{-\frac14} \cos\left( \tfrac{1}{h}\vartheta(x) \right) & \text{ if } \ell-m \text{ is even, and} \\
        \left(1-(\tfrac{x}{\eta})^2\right)^{-\frac14} \sin\left( \tfrac{1}{h}\vartheta(x) \right) & \text{ if } \ell-m \text{ is odd,} 
    \end{cases}
    \end{align}
    with error at most $\frac12 h\eta^{-3}|x|(1-x^2)(1-(\frac{x}{\eta})^2)^{-\frac74}$ in both cases.
\end{proposition}

\begin{figure}
    \centering
    \includegraphics[width=0.8\textwidth]{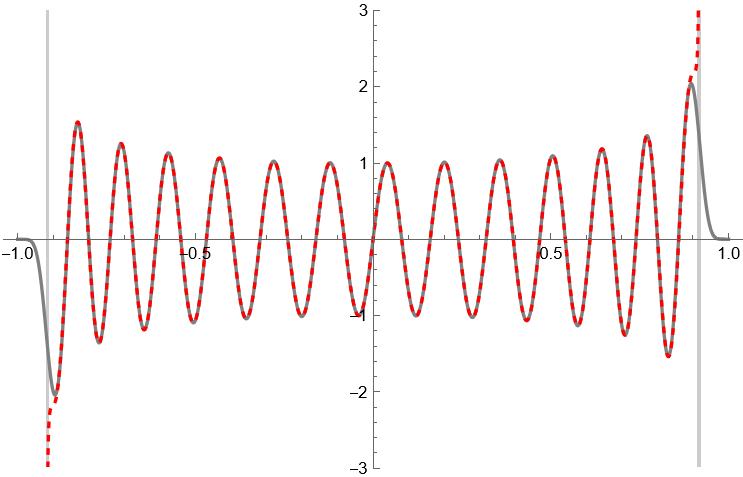}
    \caption{Plot of $\overline{P_{42}^{17}}$ (grey) and $W_{42}^{17}$ (red, dashed) as in \Cref{prop:WKB_approximation}, with vertical lines at $\pm\eta$.}
    \label{fig:legendre_approx}
\end{figure}

\begin{proof}
    Up to normalization, the associated Legendre polynomial $P_\ell^m(x)$ is the unique even (if $\ell-m$ is even) or odd (if $\ell-m$ is odd) solution on the interval $(-1,1)$ to the \emph{associated Legendre equation}  \cite[Equation 8.8.1]{AbramowitzStegun1964},
    \begin{align}
        \label{eq:associated_legendre_equation}
        (1-x^2) \, u''(x) - 2x \, u'(x) + \left(\ell(\ell+1) - \frac{m^2}{1-x^2}\right) u(x) = 0.
    \end{align}
    After a simple change of variables, we will obtain the desired asymptotic expression by Wenzel-Kramers-Brillouin approximation. 

    If $u$ solves (\ref{eq:associated_legendre_equation}), then $v(t) = u(\tanh(t))$ solves the ODE
    \begin{align}
        \label{eq:associated_legendre_equation_after_transformation}
        v''(t) + \big( \ell(\ell+1)\cosh(t)^{-2} - m \big)  v(t) = 0.
    \end{align}
    Denote $h = 1/\sqrt{l(l+1)}$, $\nu = m/\sqrt{l(l+1)}$ and $\eta = \sqrt{1-\nu^2}$ as above.
    Following the procedure detailed in \cite{Hall2013}, we now obtain the WKB approximation of the solution $v$ on the classically allowed region $\cosh(t) < \nu^{-1}$, i.e. the interval $(-\tanh^{-1}(\eta),\tanh^{-1}(\eta))$.
    Introduce $p(t) = \sqrt{\cosh(t)^{-2} - \nu^2}$, the classical momentum of a particle at position $t$, so that (\ref{eq:associated_legendre_equation_after_transformation}) can be written as
    \begin{align}
        \label{eq:associated_legendre_equation_after_transformation_2}
        v''(t) + \frac{p(t)^2}{h^2}  v(t) = 0.
    \end{align}
    For $0 \leq s \leq t < \tanh^{-1}(\eta)$, let $\Phi_0(s,t)$ and $\Phi_1(s,t)$ be the solution in $t$ of (\ref{eq:associated_legendre_equation_after_transformation_2}) with initial conditions $v(s)=1, v'(s)=0$ and $v(s)=0, v'(s)=1$, respectively. The WKB approximations to $\Phi_0(s,t)$ and $\Phi_1(s,t)$ are
    \begin{align*}
        W_0(s,t) &= \frac{A(t)}{A(s)} \cos\left(\frac1h \int_s^t p(\tau) d\tau\right) \text{ and }\\
        W_1(s,t) &= hA(t)A(s) \sin\left(\frac1h \int_s^t p(\tau) d\tau \right),
    \end{align*}
    respectively, where $A(t) = p(t)^{-\frac12}$ is another auxiliary function, chosen so that $2A'(t)p(t)+A(t)p'(t)=0$. This leads to cancellations when applying $\frac{d^2}{dt^2} + \frac{p(t)^2}{h^2}$ to $W_0(s,t)$ and $W_1(s,t)$. In fact,
    \begin{align*}
        \left(\frac{d^2}{dt^2} + \frac{p(t)^2}{h^2}\right) W_0(s,t) = \frac{A''(t)}{A(s)} \cos\left(\frac1h \int_s^t p(\tau) d\tau\right),
    \end{align*}
    which stays bounded as $h \to \infty$, and the analogous term for $W_1$ is $\mathcal O(h)$.
    
    We will now obtain a precise error bound on the approximation $W_0$. The error term $E_0(s,t) := \Phi_0(s,t) - W_0(s,t)$ satisfies $E(0)=E'(0)=0$, 
    \begin{align*}
        \left(\frac{d^2}{dt^2} + \frac{p(t)^2}{h^2}\right) E_0(s,t) = \frac{A''(t)}{A(s)} \cos\left( \frac1h \int_s^t p(\tau) d\tau\right),
    \end{align*}
    and hence can be expressed via Duhamel's formula,
    \begin{align*}
        E_0(s,t) = \frac{1}{A(s)} \int_{s}^t \Phi_1(\tau,t) A''(\tau) \cos\left(\frac1h \int_s^{\tau} p(\tau') d\tau'\right) d\tau.
    \end{align*}
    A naive bound on $\Phi_1(s,t)$ can be obtained by an energy argument as follows. For any solution $v$ of $v'' + \frac{p^2}{h^2} v = 0$, and any $t \geq 0$,
    \begin{align*}
        \frac{d}{dt} \left( v'(t)^2 + \frac{p(t)^2}{h^2} v(t)^2 \right)
        &= 2\frac{p'(t)p(t)}{h^2} v(t)^2 \leq 0.
    \end{align*}
    Thus, $|\Phi_1(s,t)| \leq \frac{h}{p(t)}$. It is easily checked that $A(t)$ is convex. Thus,
    \begin{align*}
        |E_0(s,t)| \leq \frac{h}{p(t) A(s)} \int_{s}^t A''(\tau) d\tau = h\frac{A'(t)-A'(s)}{p(t) A(s)}.
    \end{align*}
    The analogous computation for $E_1(s,t) := \Phi_1(s,t) - W_1(s,t)$ yields
    \begin{align*}
        |E_1(s,t)| \leq h^2 A(s) \frac{A'(t)-A'(s)}{p(t)}.
    \end{align*}

    Substituting $t = \tanh^{-1}(x)$, and noting that 
        \begin{align*}
            p(t) &= \sqrt{\eta^2-x^2}, & A(t) &= (\eta^2-x^2)^{-\frac14}, & A'(t) &= \frac12 \frac{x(1-x^2)}{(\eta^2-x^2)^{\frac54}},
        \end{align*}
    our WKB approximations take the form
    \begin{align*}
        W_0(0,x) &= \frac{1}{\sqrt[4]{1-(\frac{x}{\eta})^2}} \cos\left(\frac1h \int_0^x \frac{\sqrt{\eta^2-y^2}}{1-y^2} dy \right) \text{ and }\\
        W_1(0,x) &= \frac{\frac h \eta}{\sqrt[4]{1-(\frac{x}{\eta})^2}} \sin\left(\frac1h \int_0^x \frac{\sqrt{\eta^2-y^2}}{1-y^2} dy \right),
    \end{align*}
    while the error bounds become
    \begin{align*}
        |E_0(0,x)| &\leq \tfrac{1}{2} h\eta^{\frac12}\frac{|x|(1-x^2)}{(\eta^2-x^2)^{\frac74}}, & |E_1(0,x)| &\leq \tfrac{1}{2} h^2\eta^{-\frac12}\frac{|x|(1-x^2)}{(\eta^2-x^2)^{\frac74}}.
    \end{align*}
    The integral $\int_0^x \frac{\sqrt{\eta^2-y^2}}{1-y^2} dy$ can be expressed in elementary functions, yielding the expression (\ref{eq:phase}).
\end{proof}

\begin{remark}
    The change of variables $x = \tanh(t)$ arises from the observation that $(1-x^2)^2 \frac{d^2}{dx^2} - (1-x^2) 2x \frac{d}{dx}$ is the Laplace operator on the interval $(-1,1)$ equipped with the metric $\frac{1}{(1-x^2)^2}dx^2$. Its arc-length parametrization is given by $x = \tanh(t)$. The equation arising after the change of variables, 
    \begin{align*}
        v''(t) + \big( \ell(\ell+1)\cosh(t)^{-2} - m \big)  v(t) = 0,
    \end{align*}
    can be written most suggestively as $- h^2 \Delta v - \cosh(t)^{-2}v = -\nu^2 v$, using $h$ and $\nu$ as defined above. Thus, $-\nu^2 \sim -\frac{m^2}{\ell^2}$ is a negative energy level of the Schrödinger operator $-h^2 \Delta - \cosh(t)^{-2}$, while increasing $\ell$ (and hence decreasing $h \sim \frac1\ell$) corresponds to semi-classical approximation.
\end{remark}

The quality of the approximation given \Cref{prop:WKB_approximation} is demonstrated by the following corollary, which estimates the accompanying error in terms of the phase function $\vartheta(x)$.

\begin{corollary}
    \label{cor:better_error}
    Let $0 \leq m \leq \ell$ be integers, and consider $\vartheta(x)$ and $W^m_\ell(x)$ as given in \Cref{prop:WKB_approximation}. Then, for all $x \in (-\eta,\eta)$, we have the estimate
    \begin{align*}
        |W^m_\ell(x) - \overline{P_\ell^m}(x)| \leq \tfrac12 h \left( 1-(\tfrac x \eta)^2\right)^{-\frac14} (\vartheta(\eta) - \vartheta(x) )^{-1}.
    \end{align*}
\end{corollary}

\begin{proof}
    Let $e(x)= \frac12 h |x| (1-x^2)(\eta^2-x^2)^{-\frac32}$ denote the quotient of the error term given in \Cref{prop:WKB_approximation} and the amplitude $\big(1-(\frac x\eta)^2\big)^{-\frac14}$.

    The integrand in the expression $\vartheta(\eta) - \vartheta(x) = \int_x^\eta \frac{\sqrt{\eta^2-y^2}}{1-y^2} dy$ is increasing while $0 \leq y^2 \leq 2\eta^2 - 1$, and decreasing if $2\eta^2 - 1 \leq y^2 \leq \eta^2$. This suggests the following case distinction:
    
 \emph{Case 1.} $x^2 \geq 2 \eta^2-1$. In this case, the integrand is decreasing, so    \begin{align*}
        \vartheta(\eta) - \vartheta(x) &= \int_x^\eta \frac{\sqrt{\eta^2-y^2}}{1-y^2}dy
        \leq (2x)^{-1} \int_x^\eta \frac{\sqrt{\eta^2-y^2}}{1-y^2}2ydy \\
        &= (2x)^{-1} \int_{x^2}^{\eta^2} \frac{\sqrt{\eta^2-t}}{1-t}dt \leq (2x)^{-1} (\eta^2-x^2)^{\frac32}(1-x^2)^{-1},
    \end{align*}
    which implies $e(x) \leq \frac14 h (\vartheta(\eta) - \vartheta(x))^{-1}$.

 \emph{Case 2.} $x < \sqrt{2\eta^2-1}$. Here we estimate $\eta^2-y^2 \leq 1-y^2$ and obtain 
    \begin{align*}
        \vartheta(\eta) - \vartheta(x) &\leq \int_x^\eta (1-y^2)^{-\frac12}dy \leq  (2x)^{-1} \int_x^\eta (1-y^2)^{-\frac12} 2ydy \\
        &=  (2x)^{-1} \int_x^\eta (1-t)^{-\frac12}dt \leq x^{-1} \left( (1-x^2)^{\frac12} - (1-\eta^2)^{\frac12}\right).
    \end{align*}
    The case assumption is equivalent to $1-x^2 \geq 2 (1-\eta^2)$. It is elementary to check that $(1-x^2)^{\frac12} - (1-\eta^2)^{\frac12} \leq (\eta^2-x^2)^{\frac32}(1-x^2)^{-1}$ under this condition. Thus, $e(x) \leq \frac12 h (\vartheta(\eta) - \vartheta(x))^{-1}$. Multiplying by the amplitude $\big(1-(\frac x\eta)^2\big)^{-\frac14}$ yields the claimed bound on the error.
\end{proof}

The associated Legendre polynomial $P_\ell^m$ has $\ell-m$ simple zeros in the interval $(-1,1)$, which are arranged symmetrically with respect to the origin. Thus, the interval $(0,1)$ contains $\lfloor \frac{\ell-m}{2} \rfloor$ of these zeroes. In fact, all nontrivial zeros of $P_\ell^m$ lie within the interval $(-\eta,\eta)$, since the Legendre differential equation implies $P_\ell^m$ cannot have local extrema in $(-1,-\eta)\cup(\eta,1)$.

\begin{corollary}
\label{cor:zeros}
    The $k^{th}$ zero of $P^m_\ell$ in the interval $(0,1)$, which we shall denote by $\xi_\ell^m(k)$, is approximated well by the unique solution to the equation
    \begin{align}
        \label{eq:kth_zero_approximation}
        \tan^{-1}\left(\tfrac{\frac{x}{\eta}}{\sqrt{1-(\frac{x}{\eta})^2}}\right) - \nu \, \tan^{-1}\left(\nu \, \tfrac{\frac{x}{\eta}}{\sqrt{1-(\frac{x}{\eta})^2}}\right) = h \left(k-\tfrac{1+(-1)^{\ell-m}}{4}\right) \pi.
    \end{align}
More precisely, $|\vartheta(\xi_\ell^m(k)) - h\left(k-\tfrac{1+(-1)^{\ell-m}}{4}\right)\pi| \leq \frac h2(\ell-m-2k+0.42)^{-1}$.
\end{corollary}

\begin{proof}
Let us show first that (\ref{eq:kth_zero_approximation}) indeed has a unique positive solution if and only if $1 \leq k \leq \lfloor \frac{\ell-m}{2} \rfloor$. This follows from the fact that $\frac{x}{\eta}\big(1-(\frac{x}{\eta})^2\big)^{-\frac12}$ and $\tan^{-1}(y)-\nu \tan^{-1}(\nu y)$ are bijections from $(0,\eta)$ to $(0,\infty)$ and from $(0,\infty)$ to $(0,(1-\nu) \tfrac{\pi}{2})$, respectively. Thus, (\ref{eq:kth_zero_approximation}) has a unique solution in $(0,\eta)$ if and only if $2k-\tfrac{1+(-1)^{\ell-m}}{2} \leq h^{-1}(1-\nu) = \sqrt{\ell(\ell+1)} - m$. Since $\ell \leq \sqrt{\ell(\ell+1)} < \ell + \frac12$, the claim follows.

Let $\zeta(k)$ denote the number given by (\ref{eq:kth_zero_approximation}). It is of course just the $k^{th}$ zero of the approximation $W^m_\ell$ for $\overline{P_\ell^m}$ given in \Cref{prop:WKB_approximation}.
Once we show, via the intermediate value theorem, that $\overline{P_\ell^m}$ has a zero near $\zeta(k)$ for every $k \in \{1,\ldots,\lfloor \frac{\ell-m}{2} \rfloor\}$, the corollary is established, since all zeros of $\overline{P_\ell^m}$ are then accounted for by symmetry.

Let $a(x)=\frac12 h \eta^{-3} |x| (1-x^2)(1-(\frac{x}{\eta})^2)^{-\frac32}$ denote the ratio of the error term given in \Cref{prop:WKB_approximation} and the amplitude of $W^m_\ell(x)$. Assume $\ell-m$ is even for simplicity; the odd case is analogous. Let $\zeta^{\pm}(k)$ denote the positive solutions to the equation $\cos(\frac{1}{h}\vartheta(x)) = \pm a(x)$ in ascending order. By construction, $\overline{P_\ell^m}(\zeta^+(k)) > 0$ and $\overline{P_\ell^m}(\zeta^-(k)) < 0$, so the interval $(\zeta^+(k), \zeta^-(k))$ contains some zero of $\overline{P_\ell^m}$.

Our first goal is to establish that all zeros $\zeta(k)$ of $\cos\left(\frac1h\vartheta(x)\right)$ are in fact contained in an interval of this kind. Since $\sqrt{\ell(\ell+1)} - \ell > 0.42$ for all $\ell \geq 1$, at the $k^{th}$ local extremum of the function $\cos\left(\frac1h\vartheta(x)\right)$ in the interval $(0,\eta)$ the inequality $\frac{1}{h}\left(\vartheta(\eta) - \vartheta(x)\right) \geq 0.21 \, \pi + (\lfloor \frac{\ell-m}2\rfloor - k)\pi$ holds. \Cref{cor:better_error} says $a(x) < \frac h2 \left(\vartheta(\eta) - \vartheta(x)\right)^{-1}$, so $a(x) < \frac{1}{0.42 \pi} < 1 = |\cos\left(\frac1h\vartheta(x)\right)|$ at any such point. Between two consecutive local extremas of $\cos\left(\frac1h\vartheta(x)\right)$, the graph of this function thus intersects the graphs of $\pm a(x)$ at least once. We conclude that for \emph{any} choice of integers $\ell \geq 1$, $0 \leq m \leq \ell$ and $1 \leq k \leq \lfloor \frac{\ell-m}{2} \rfloor$ the interval $(\zeta^+(k), \zeta^-(k))$ exists and contains both $\zeta(k)$ and $\xi_\ell^m(k)$.

For the precise error estimate, assume first that $\zeta(k) < \xi_\ell^m(k) < \zeta^+(k)$. Because $\zeta^+(k)$ comes before the next critical point of $\cos\left(\frac1h\vartheta(x)\right)$, we can estimate $\cos\left(\frac1h\vartheta(\zeta^+(k))\right)$ from below by $\frac2\pi (\zeta^+(k) - \zeta(k))$, and obtain
\begin{align*}
    \vartheta\left(\xi_\ell^m(k)\right) - \vartheta(\zeta(k)) \leq \vartheta\left(\zeta^+(k)\right) - \vartheta(\zeta(k)) \leq \frac{\pi}{2} h a(\zeta^+(k)).
\end{align*}
From \Cref{cor:better_error}, we know $a(\zeta^+(k)) < \frac h2\left(\vartheta(\eta)-\vartheta(\zeta^+(k))\right)^{-1}$, and since $\frac1h\vartheta(\zeta^+(k)) < k\pi$ and $\frac1h\vartheta(\eta)>(\ell-m+0.42)\frac\pi2$, it follows that
\begin{align*}
    \vartheta\left(\xi_\ell^m(k)\right) - \vartheta(\zeta(k)) \leq \frac h2 (\ell-m-2k+0.42)^{-1}.
\end{align*}
The estimates in the case $\zeta^-(k) < \xi_\ell^m(k) < \zeta(k)$ are exactly analogous.
\end{proof}

%% file: Smooth_Functions.tex
    It remains to prove \Cref{thm:general}, which we will achieve by a coordinate change, an application of Mather's division theorem and finally a straightforward power series argument.
    
    After a linear coordinate change which simultaneously diagonalizes $a(0)$ and $\mathrm{Hess}_u(0)$, we may assume $a(0)_{jk} = \delta_{jk}$ (the Kronecker delta) and $u(x) = \sum_{j=1}^d \lambda_j x_j^2 + \mathcal O(|x|^3)$. A second coordinate change, guaranteed by the Morse lemma, achieves $u(x) = \sum_{j=1}^d \lambda_j x_j^2$ while preserving $a(0)_{jk} = \delta_{jk}$.
    
    Any germ $v \in C^\infty(\mathbb R^d;0)$ vanishing on the zero set of $u$ is in fact divisible by $u$, i.e.\ there exists $q \in C^\infty(\mathbb R^d; 0)$ with $v = qu$. This follows from Mather's division theorem (see e.g.\ \cite[Chapter IV, Theorem 2.1]{GolubitskyGuillemin1973}): We can write 
    \begin{align*}
        v(x) = u(x) q(x) + r_0(x_1,\ldots,x_{d-1}) + r_1(x_1,\ldots,x_{d-1})x_d,
    \end{align*}
    for $q \in C^\infty(\mathbb R^d;0)$ and $r_0$, $r_1 \in C^\infty(\mathbb R^{d-1};0)$. As $Lu = 0$, $\sum_{j=1}^d \lambda_j = 0$, implying $\lambda_d > 0$ and that the set $\{y \in \mathbb R^{d-1}: \sum_{j=1}^{d-1} \lambda_j y_j^2 < 0\}$ is non-empty. For any $x \in \mathbb R^d$ such that $x' = (x_1,\ldots,x_{d-1})$ lies in this non-empty open cone, $v(x',x_d)$ vanishes for two values of $x_d$, hence $r_0(x') = r_1(x') = 0$. It follows that, $r_0(x')+r_1(x')x_d$ vanishes to infinite order on $u^{-1}(\{0\})$, so $r_0(x') + r_1(x')x_d$ can be absorbed into $q(x)$.

    Let $A_{jk}, B_j, C$ and $Q$ denote the infinite Taylor series of $a_{jk}, b_j, c$ and $q$ at the origin, and decompose them into their homogeneous parts,
    \begin{align*}
        A_{jk} &= \delta_{jk} + \sum_{\ell = 1}^\infty a_{jk;\ell},  & B_{j} &= \sum_{\ell = 0}^\infty b_{j;\ell}, & C &= \sum_{\ell = 0}^\infty c_{\ell}, & Q &= \sum_{\ell = 0}^\infty q_{\ell}.
    \end{align*}
    After subtracting a multiple of $u$ from $v$, we may assume $q_0 = 0$. On the level of formal power series, the equation $Lv = 0$ means that
    \begin{align}
        \begin{split}
            \label{eq:recursion}
        \Delta\left(u q_\ell \right) = &- \sum_{\ell'=1}^\ell \sum_{j,k=1}^d a_{jk;\ell'} \frac{\partial^2}{\partial x_j \partial x_k}(u q_{\ell - \ell'}) \\ &- \sum_{\ell'=1}^\ell \sum_{j=1}^d b_{j;\ell'-1} \frac{\partial}{\partial x_j}(u q_{\ell - \ell'}) - \sum_{\ell'=2}^\ell c_{\ell'-2} u q_{\ell-\ell'},
        \end{split}
    \end{align}
    for all $\ell \in \mathbb N$. For $\ell \not\in \{0,d\}$, the map $q_\ell \to \Delta\left(u q_\ell \right)$ is an invertible linear map by \Cref{thm:main}. We therefore recursively obtain $q_0 = q_1 = \ldots = q_{d-1} = 0$. The term $q_d$ may be any multiple of $x_1\ldots x_d$. The remaining terms are then uniquely determined again by the recursion (\ref{eq:recursion}). We conclude that the space of formal power series solutions $q$ to $L(uq) = 0$ is spanned by the constant function and the unique solution to (\ref{eq:recursion}) with $q_d = x_1\ldots x_d$.
    
    The space of germs of smooth functions $q$ which solve $L(uq) = 0$ is thus at most two-dimensional, since, due to Aronszajn \cite{Aronszajn1957}, no nontrivial solution to $Lv = 0$ can vanish to infinite order at $0$. It could be, of course, that the second solution on the level of power series does not correspond to a smooth solution at all --- not even in the case of analytic coefficients, since the construction involves repeatedly inverting the presumably often ill-conditioned linear map $q_\ell \to \Delta( u q_\ell)$. \qed